\documentclass[11pt]{article}

\usepackage[margin=1 in]{geometry}  
\usepackage{natbib}  

\usepackage{xcolor}         % colors
\usepackage{algorithm}
\usepackage{algpseudocode}
\usepackage{lipsum}
\usepackage[utf8]{inputenc} % allow utf-8 input
\usepackage[T1]{fontenc}    % use 8-bit T1 fonts
\usepackage{hyperref}       % hyperlinks
\usepackage{url}            % simple URL typesetting
\usepackage{booktabs}       % professional-quality tables
\usepackage{amsfonts}       % blackboard math symbols
\usepackage{nicefrac}       % compact symbols for 1/2, etc.
\usepackage{microtype}      % microtypography
\usepackage{xcolor}
\usepackage{ upgreek }
\usepackage{graphicx}
\usepackage{subfigure}
\usepackage{booktabs} 
\usepackage{hyperref}
\usepackage{times}		
\usepackage[T1]{fontenc}
\usepackage[utf8]{inputenc} 
\usepackage{amssymb,amsmath,amscd,amsfonts,amsthm,bbm,mathrsfs,yhmath}

\usepackage{hyperref}
\usepackage{caption}
\usepackage{graphics}
\usepackage{mathtools}
\usepackage{cleveref}
\usepackage{comment}

\usepackage{paralist,enumitem}                                       
\usepackage{pdfpages}

\usepackage[textwidth=2cm, textsize=footnotesize]{todonotes}  

\newtheorem{theorem}{Theorem}
\newtheorem{lemma}{Lemma}
\newtheorem{corollary}{Corollary}
\newtheorem{proposition}{Proposition}
\newtheorem{definition}{Definition}
\newtheorem{remark}{Remark}

\newtheorem{assumption}{Assumption}

\newcommand{\rmd}{{\mathrm d}}

\newcommand{\cO}{{\mathcal O}}

\newcommand{\cN}{{\mathcal N}}

\newcommand{\cP}{{\mathcal P}}

\newcommand{\cH}{{\mathcal H}}

\newcommand{\R}{{\mathbb R}}

\newcommand{\K}{{K}} 
\newcommand{\MM}{{M}}

\newcommand{\KL}{\mathop{\mathrm{KL}}\nolimits}

\newcommand{\op}{\mathop{\mathrm{op}}\nolimits}

\newcommand{\Stein}{ MStein}

\newcommand\bX{\boldsymbol X}

\newcommand\NN{\mathbb N}
\newcommand\RR{\mathbb R}
\newcommand\EE{\mathbb E}

\newcommand{\norm}[1]{\left\| #1 \right\|}
\newcommand{\normsq}[1]{\left\| #1 \right\|^2}
\newcommand{\inner}[2]{\left< #1 , #2 \right>}

\title{A Note on the Convergence of  Mirrored Stein Variational Gradient Descent under $(L_0,L_1)-$Smoothness Condition}

\author{Lukang Sun\\KAUST\\	\texttt{lukang.sun@kaust.edu.sa} \thanks{King Abdullah University of Science and Technology, Thuwal, Saudi Arabia} \and Peter Richt\'{a}rik \\KAUST\\	\texttt{peter.richtarik@kaust.edu.sa} }

\begin{document}

	\maketitle

	\begin{abstract}
		In this note, we establish a descent lemma for the population limit Mirrored Stein Variational Gradient Method~(MSVGD). This descent lemma does not rely on the path information of MSVGD but rather on a simple assumption for the mirrored distribution $\nabla\Psi_{\#}\pi\propto\exp(-V)$. Our analysis demonstrates that MSVGD can be applied to a broader class of constrained sampling problems with non-smooth $V$. We also investigate the complexity of the population limit MSVGD in terms of dimension $d$.

%		Stein Variational Gradient Descent (SVGD) is an important alternative to the Langevin-type algorithms for sampling from a probability distribution of the form $\pi(x) \propto \exp(-V(x))$. 
%		In the existing theory of Langevin-type algorithms and SVGD,  the potential function $V$ is often assumed to be $L$-smooth. 
%		However, this restrictive condition excludes a large class of functions such as polynomials of degree greater than $2$. 
%		Our paper studies the convergence of the SVGD algorithm in population limit for distributions with $(L_0,L_1)$-smooth potentials. 
%		This relaxed smoothness assumption was introduced by \cite{zhang2019gradient} for gradient clipping algorithms in optimization.
%		With the help of trajectory-independent auxiliary conditions, we provide a descent lemma establishing that the algorithm decreases the KL-divergence at each iteration and prove a complexity bound for SVGD in population limit in terms of the Stein  Fisher information.
		
	\end{abstract}
	
	\section{Introduction}
	The constrained optimization problem
	\begin{equation}
		\min_{\theta\in\Omega}F(\theta),
	\end{equation}
	where $\Omega$ is some constrained convex region in $\R^d$, is of importance in practice. The above constrained optimization problem can be solved by constructing a differentiable one to one map $T:\R^d\to\Omega$, then by compositing $T$, we get the unconstrained 
	optimization problem
		\begin{equation}
		\min_{x\in\R^d}F\circ T(x),
	\end{equation}
which can be solved efficiently by gradient based approaches. One method based on such kind of transformation is the renowned mirror descent algorithm~\cite{nemirovskij1983problem}: 
	\begin{equation}
		\nabla\Psi(\theta_{n+1})=\nabla\Psi(\theta_n)-\gamma\nabla F(\theta_n),
	\end{equation}
	where $\Psi(\cdot):\Omega\to\R$ is the mirror function, $\nabla\Psi(\cdot):\Omega\to\R^d$ is called the mirror map and $\gamma$ is the step-size. This algorithm is based on the discretization of the gradient flow of $F\circ\nabla\Psi^{-1}(x)$:
	\begin{equation}
		\frac{d}{dt}x_t=-\nabla F\circ\nabla\Psi^{-1}(x_t)=-\nabla^2\Psi^{-1}(x_t)\nabla F(\nabla\Psi^{-1}(x_t)).
	\end{equation}
%\cite{hsieh2018mirrored,zhang2020wasserstein,ahn2021efficient} used the Mirrored Langevin algorithm to solve the constrained sampling problem: sample points from $\pi(\theta)\propto\exp(-F(\theta)),\theta\in\Omega$. 
%Like the unadjusted Langevin algorithm,  the Mirrored Langevin algorithm can be naturally treated as the discretization of the Wasserstein gradient flow of $\KL(\cdot\mid\bar{\pi})$, where $\bar{\pi}(x):=\nabla\Psi_{\#}\pi (x)$ and $\bar{\pi}(x)\propto\exp(-V(x)),x\in\R^d$.

Inspired by this, most recently \cite{shi2021sampling} proposed a variant of Stein Variational Gradient Descent~(SVGD) called Mirrored Stein Variational Gradient Descent~(MSVGD) to solve the constrained sampling problem: to sample points from distribution $\pi\propto\exp(-F)$ with support on $\Omega$. This method is equivalent to SVGD applied on the mirrored distribution $\bar{\pi}=\nabla\Psi_{\#}\pi$, however it doesn't require the computation related to $\bar{\pi}$. In \cite{shi2021sampling}, under the $L$-smoothness condition of $V$ they showed that MSVGD decreases the KL divergence following the method from \cite{liu2017stein}. In this note, inspired by \cite{sun2022convergence} we reanalyze MSVGD under a weaker smoothness condition called $(L_0,L_1)$-smoothness originally from \cite{zhang2019gradient}. This condition allows $V$ to grow like a polynomial with degree larger than 2, so it include a large class of non-smooth functions into it. One more advantage of our analysis is that we don't require the path information of MSVGD like in \cite{shi2021sampling}, so we can establish a complexity analysis of the population limit MSVGD in terms of the desired accuracy $\varepsilon$ and dimension $d$.

	\subsection{Related work}
	
	A parallel approach to solve the constrained sampling problem is by the Mirrored Langevin algorithm \cite{hsieh2018mirrored,ahn2021efficient,zhang2020wasserstein}.
	The work of \cite{liu2016stein} introduced SVGD as a sampling method, since then several variants of SVGD have been considered, these 
	include: random batch method SVGD~(RBM-SVGD)~\cite{li2020stochastic}, matrix kernel SVGD~\cite{wang2019stein}, Newton version SVGD~\cite{
		detommaso2018stein}, Stochastic SVGD~\cite{gorham2020stochastic}, Mirrored SVGD~\cite{shi2021sampling}  etc. However, most of the theoretical understanding of SVGD is still limited to continuous time approximation of SVGD or population limit SVGD : ~\cite{lu2019scaling} studied the scaling limit of SVGD in continuous time,  ~\cite{duncan2019geometry} studied the geometry related to 
	SVGD,~\cite{liu2017stein} first built a convergence result of SVGD  in the population limit, ~\cite{korba2020non} established a descent lemma for 
	SVGD in population limit in terms of Kullback-Leibler divergence however their analysis relied on the path information of SVGD which is unknown 
	beforehand, ~\cite{salim2021complexity} provided a clean analysis based on the work of ~\cite{korba2020non}, they assumed $\pi$ satisfies Talagrand's $\text{T}_1$\footnote{see \Cref{eq:Tp} with $p=1$.}inequality then they resolved the problem of relying on  path information and got a complexity bound for SVGD in terms of the desired 
	accuracy $\varepsilon$ and dimension $d$ which is first shown up in the analysis of SVGD. \cite{sun2022convergence}
	introduced the $(L_0,L_1)$-smoothness condition to analyze SVGD and under a generalized $\text{T}_p$ inequality they also got a complexity bound for SVGD in terms of $\varepsilon$ and $d$.
	
	\subsection{Contributions}
	
	The contributions of this work can be listed in three folds:
	
	\begin{itemize}
		\item We study the population limit MSVGD under a novel smoothness criterion that was utilized to analyze SVGD in \cite{sun2022convergence}. This smoothness condition originally comes from the optimization literature and allows the local smoothness constant to increase with the gradient norm and so it is strictly weaker than the original Lipschitz gradient assumption.
		
		\item Under this smoothness assumption, we build a descent lemma for population limit MSVGD. Our result is complete in that we only make assumptions on the mirrored target distribution $\bar{\pi}$ but not on the path information like in \cite{shi2021sampling}.
		
		\item We provide a complexity bound for the population limit MSVGD under this smoothness condition.
	\end{itemize}

	\subsection{Paper structure}
	
	The paper is organized as follows.  \Cref{sec:prelimi} introduces the background needed on optimal transport, reproducing kernel Hilbert space~(RKHS) and Mirrored Stein variational gradient descent~(MSVGD); \Cref{sec:assump} presents the assumptions needed in our analysis; \Cref{sec:maintheory} shows theoretical results on the population limit MSVGD under assumptions from \Cref{sec:assump}; We conclude this work in \cref{sec:conclusion} and for the missing proofs  and explanations, please refer to Appendix \ref{appendix}.

	\section{Preliminaries}\label{sec:prelimi}
	
	\subsection{Notations}
	
	For every $x=(x_1,\ldots,x_d)^\top,y=(y_1,\ldots,y_d)^\top\in\R^d$, we will denote $\norm{x}^p:=\left(\sum_{i=1}^d|x_i|^2\right)^{\frac{p}{2}}, \forall p\geq1$, $\norm{x}=\sqrt{\sum_{i=1}^d|x_i|^2}$ and $\inner{x}{y}=\sum_{i=1}^dx_iy_i$. We will use $\nabla^2V(x)$ to denote the Hessian of potential function $V(\cdot):\R^d\to\R$ at point $x$ and $\norm{\nabla^2V(x)}_{op}$ to denote the operator norm of matrix $\nabla^2V(x)$. The Jacobian matrix of a vector valued function 
	$h(\cdot) = (h_1(\cdot),\ldots,h_d(\cdot))^\top$ with each $h_i(\cdot):\R^d\to\R$ is a $d\times d$ matrix defined as
	\begin{equation*}
		J h(x) := \big({ \frac{\partial h_i}{\partial x_j}}(x)\big)_{i=1,j=1}^{d,d}.
	\end{equation*}
	We will use $\norm{J h(x)}_{HS}$ to denote the Hilbert-Schmidt norm of matrix $Jh(x)$, that is 
	\begin{equation*}
		\norm{J h(x)}_{HS}=\sqrt{\operatorname{tr}\left(Jh(x)^\top Jh(x)\right)}.
	\end{equation*}
	The notations $\Omega$ and
	$\mathcal{O}$ follow the convention in complexity analysis, i.e., $f=\Omega(g)$ means $f \gtrsim g$, and $f=\mathcal{O}(g)$
	means $f \lesssim g$. Notations $\tilde{\Omega},\tilde{\mathcal{O}}$ are used in a similar way, but they don't take into account the dependence of parameter except $p,d,\lambda,C_{\bar{\pi},p}$ and $\varepsilon$.  ${\rm I}_d$ is denoted as the $d\times d$ identity matrix.
	
	\subsection{Optimal Transport}\label{subsec:ot}
	Let $\mathcal{X}_1,\mathcal{X}_2$ be two measurable spaces, and denote $\cP(\mathcal{X}_1),\cP(\mathcal{X}_2)$ as the sets of all Borel probability measures on $\mathcal{X}_1,\mathcal{X}_2$ respectively. Given a measurable map $T: \mathcal{X}_1 \to \mathcal{X}_2$ and $\mu \in \mathcal{P}(\mathcal{X}_1)$, we denote by $T_{\#} \mu\in\cP(\mathcal{X}_2)$ the pushforward measure of $\mu$ by $T$, characterized by the transfer lemma $\int_{\mathcal{X}_1} \phi(T(x)) d \mu(x)=\int_{\mathcal{X}_2} \phi(y) d T_{\#} \mu(y)$, for any measurable and bounded function $\phi$ defined on $\mathcal{X}_2$. 
	We denote by $\cP_p(\R^d)$ with $p\geq 1$ the set of Borel measures $\mu$ on $\R^d$ with finite $p$-th absolute moment, that is  $\int\norm{x}^pd\mu(x)<+\infty$. For every $\mu, \nu\in\cP_p(\R^d)$, we denote $\Gamma(\mu,\nu)$ as the set of all the coupling measures between $\mu$ and $\nu$ on $\R^{d}\times\R^d$, that is for any
	$\gamma\in\Gamma(\mu,\nu)$, 
	we have $\mu={T_1}_{\#}\gamma,\nu={T_2}_{\#}\gamma$ 
	with $T_1(x,y)=x, T_2(x,y)=y$.
	The Wasserstein distance $W_p$ between $\mu$ and $\nu$ is defined by 
	\begin{equation}
		\label{eq:wstdis}
		W_p(\mu,\nu)=\left(\inf_{\gamma\in\Gamma(\mu,\nu)}\int \norm{x-y}^pd\gamma(x,y)\right)^{\frac{1}{p}},
	\end{equation}
	we can show that $W_p(\cdot,\cdot)$ is a metric on $\cP_p(\R^d)$ and so $\left(\cP_p(\R^d),W_p\right)$ is a metric space, see \cite{ambrosio2008gradient}.  We define the Kullback-Leibler~(KL) divergence of $\mu$ with respect to $\pi$ as 
	\begin{equation}
		\label{eq:KLdiv}
		\KL(\mu\mid\pi):=\begin{cases}
			\int\log(\frac{\mu}{\pi})(x)d\mu(x) & \text{$\mu$ is abosolutely continuous with respect to $\pi$} \\
			\infty                              & \text{else}
		\end{cases}.
	\end{equation}
	
	\subsection{Mirrored Sein Variational Gradient Descent~(MSVGD)}
	\begin{algorithm}[t!]
		\caption{Mirrored Sein Variational Gradient Descent~\cite{shi2021sampling}}\label{alg:1}
		\begin{algorithmic}[1]
			\State {\bfseries Input:} density $p$ on $\Omega$, kernel $k(\cdot,\cdot)$, mirror function $\Psi$, particles $\left(\theta_{0}^{i}\right)_{i=1}^{N} \subset \Omega$, step sizes $\gamma$
			\State{\bfseries Init:} $x_{0}^{i} \leftarrow \nabla \Psi\left(\theta_{0}^{i}\right)$ for $i \in[N]$
			\For{$k=0,1,\cdots,n$ }
			\State $x_{k+1}^i\leftarrow x_k^i+\gamma \frac{1}{N}\sum_{j=1}^N k(\theta_k^i,\theta_k^j)\nabla^2\Psi(\theta_k^j)^{-1}\nabla_{\theta}\log(\pi)(\theta_k^j)+\nabla_{\theta}\cdot \left(\nabla^2\Psi(\theta)^{-1}k(\theta_k^i,\theta)\right)\mid_{\theta=\theta_k^j}$ \\for $i\in [N]$
			\State $\theta_{k+1}^i\leftarrow \nabla\Psi^*(x_{k+1}^i)$ for $i\in [N]$
			\EndFor
			\State {\bfseries Return:} $\left(\theta_{n+1}^i\right)_{i=1}^N$.
		\end{algorithmic}
	\end{algorithm}
We First introduce the background on the Reproducing Kernel Hilbert Space~(RKHS).
Let $\cH_0$ denote a Reproducing Kernel Hilbert Space (RKHS) with kernel $k(\cdot,\cdot):\Omega\times\Omega\to \R$. For every $f,g\in\cH_0$, denote the inner product between $f$ and $g$ on $\cH_0$ as $\langle f,g\rangle_{\cH_0}$ and the norm of $f$ as $\norm{f}_{\cH_0}:=\sqrt{\langle f,f\rangle_{\cH_0}}$. The most important property on $\cH_0$ is the so called reproducing property: $f(\theta)=\langle f(\cdot),k(\theta,\cdot)\rangle_{\cH_0}, \forall f\in\cH_0$.
The
$d$-fold Cartesian product of $\cH_0$ is denoted by $\cH$, then the inner product on $\cH$ is given by $\inner{f}{g}_{\cH}:=\sum_{i=1}^d\inner{f_i}{g_i}_{\cH_0}$, here $f=(f_1,\ldots,f_d)^\top,g=(g_1,\ldots,g_d)^\top\in\cH$~(that is each $f_i,g_i\in\cH_0$). 

Let $\Omega$ be some closed convex domain in $\R^d$ and $\Psi(\cdot):\Omega\to (-\infty,+\infty]$ be some proper, closed and 
strongly ${\K}$-convex function defined on $\Omega$. Define $\Psi^{*}(x):=\sup_{\theta\in\Omega}x^{\top}\theta-\Psi(\theta)$ the convex conjugate of $\Psi$, we always assume the domain of $\Psi^*$ is $\R^d$. It is well known that $\Psi^*$ is convex on $\R^d$ and the inverse map of $\nabla\Psi$ is $\nabla\Psi^*$, that is $\nabla_{x}\Psi^*(x)\mid_{x=\nabla\Psi(\theta)}=\theta$, so it is easy to verify that $\norm{\nabla^2\Psi^*}_{\op}=\norm{\nabla^2\Psi^{-1}}_{op}\leq\frac{1}{K}$~(see \cite{beck2017first}). One important example in application is
 $\Omega=\left\{\theta\in\R^d:\sum_{i=1}^d\theta_i\leq 1,\theta_i\geq 0, \forall i\in[d]\right\}$ and $\Psi(\theta)=\sum_{i=1}^d\theta_i\log(\theta_i)+\left(1-\sum_{i=1}^d\theta_i\right)\log(1-\sum_{i=1}^d\theta_i)$. It is easy to verify that $\Psi$ is $1$-strongly convex and $\operatorname{dom}(\nabla\Psi^*)=\R^d$. 
 
 With the map $\Psi$, we can now turn the constrained sampling problem $\pi(\theta):=\exp(-F(\theta)),\theta\in\Omega$ into unconstrained sampling problem $\bar{\pi}(x):=\exp(-V(x)),x\in\R^d$ with $\bar{\pi}=\nabla\Psi_{\#}\pi$. The idea behind the MSVGD is first to sample point $x\sim\bar{\pi}$ through ordinaly SVGD method, then by map $\nabla\Psi^*$ we get point $\theta=\nabla\Psi^*(x)\sim\pi$. 

	For any $\mu\in\cP(\Omega)$, we have 
	\begin{equation}\label{eq:gmu}
	\int_{\theta\in\Omega}k(\theta,\cdot)\nabla^2\Psi(\theta)^{-1}\nabla_{\theta}\log(\frac{\mu}{\pi})(\theta)d\mu(\theta)\in\cH,
	\end{equation}
 we can now give the definition of the Mirrored Stein Fisher information:
	\begin{definition}
		The Mirrored Stein Fisher information  of $\mu$ relative to $\pi$ with mirror map $\Psi$ is given by
		\begin{equation*}
			I_{\Stein}\left(\mu\mid\pi\right):= \int_{\theta\in\Omega}\int_{\theta'\in\Omega}k(\theta,\theta')\inner{\nabla^2\Psi(\theta)^{-1}\nabla_{\theta}\log(\frac{\mu}{\pi})(\theta)}{\nabla^2\Psi(\theta')^{-1}\nabla_{\theta'}\log(\frac{\mu}{\pi})(\theta')}d\mu(\theta)d\mu(\theta').
		\end{equation*}
	\end{definition}
	On the one hand, when $\Omega=\R^d$ and $\Psi(\theta)=\frac{\norm{\theta}^2}{2}$, this Mirrored Stein Fisher information will be reduced to Stein Fisher information \cite{duncan2019geometry}; on the other hand, it can be seen as the Stein Fisher information of $\nabla\Psi_{\#}\mu$ relative to $\nabla\Psi_{\#}\pi$ with kernel $k(\nabla\Psi^*(\cdot),\nabla\Psi^*(\cdot))$, see \Cref{lem:oonnee}.
	\begin{lemma}\label{lem:oonnee}
		Let $\bar{\mu}:=\nabla\Psi_{\#}\mu$ and $\bar{\pi}:=\nabla\Psi_{\#}\pi$, then we have
		\begin{equation}\label{eq:noisy}
		\begin{aligned}
			g_{\bar{\mu}}(\cdot)&:=P_{\bar{\mu}}\nabla\log(\frac{\bar{\mu}}{\bar{\pi}})(\cdot)
				=\int_{x\in\R^d}k(\nabla\Psi^*(x),\cdot)\nabla_x\log(\frac{\bar{\mu}}{\bar{\pi}})(x)d\bar{\mu}(x)=\int_{\theta\in\Omega}k(\theta,\cdot)\nabla^2\Psi(\theta)^{-1}\nabla_{\theta}\log(\frac{\mu}{\pi})(\theta)d\mu(\theta)
		\end{aligned}
		\end{equation}
	and
	\begin{equation}
			I_{\Stein}\left(\mu\mid\pi\right)=\int_{x\in\R^d}\int_{y\in\R^d}k(\nabla\Psi^*(x),\nabla\Psi^*(y))\inner{\nabla_{x}\log(\frac{\bar{\mu}}{\bar{\pi}})(x)}{\nabla_{y}\log(\frac{\bar{\mu}}{\bar{\pi}})(y)}d\bar{\mu}(x)d\bar{\mu}(y)=\norm{g_{\bar{\mu}_n}}_{\cH}^2.
	\end{equation}
	\end{lemma}
	
	The form of $g_{\bar{\mu}}(\cdot)$ in the right hand side of \eqref{eq:noisy} is not attainable in practice, since it require the gradient of $\log(\frac{\mu}{\pi})$, however with a non-demanding assumption, we have the following representation of $g_{\bar{\mu}}$.
		\begin{lemma}\label{lem:ttwwoo}
		Assume for any $\theta'\in\Omega$, $k(\theta,\theta')\nabla^2 \Psi(\theta)^{-1}\mu(\theta)$ vanishes when $\theta\to\partial\Omega$, then we have 
		\begin{equation}
			\begin{aligned}
				g_{\bar{\mu}}(\cdot)&=\int_{\theta\in\Omega}k(\theta,\cdot)\nabla^2\Psi(\theta)^{-1}\nabla_{\theta}\log(\frac{\mu}{\pi})(\theta)d\mu(\theta)\\
				&=-\int_{\theta\in\Omega}k(\theta,\cdot)\nabla^2\Psi(\theta)^{-1}\nabla_{\theta}\log(\pi)(x)d\mu(\theta)-\int_{\theta\in\Omega}\nabla_{\theta}\cdot\left(k(\theta,\cdot)\nabla^2\Psi(\theta)^{-1}\right)d\mu(\theta).
			\end{aligned}
		\end{equation}
	\end{lemma}
When kernel $k(\cdot,\cdot)$ and density of $\mu$ is bounded, this vanishing condition of $k(\theta,\cdot)\nabla^2\Psi(\theta)^{-1}\mu(\theta)$ at boundary $\partial\Omega$ is satisfied by the example $\Omega=\left\{\theta\in\R^d:\sum_{i=1}^d\theta_i\leq 1,\theta_i\geq 0, \forall i\in[d]\right\}$ and $\Psi(\theta)=\sum_{i=1}^d\theta_i\log(\theta_i)+\left(1-\sum_{i=1}^d\theta_i\right)\log(1-\sum_{i=1}^d\theta_i)$. 
	
	With all the preparation, we can show what we mean by population limit~(or infinite particle regime) MSVGD, that is \begin{equation}
		\begin{cases}
			\bar{\mu}_{n-1}=\nabla\Psi_{\#}\mu_{n-1}\\
			\bar{\mu}_n=\left({\rm I}_d-\gamma g_{\bar{\mu}_{n-1}}\right)_{\#}\bar{\mu}_{n-1}\\
			\mu_n=\nabla\Psi^*_{\#}\bar{\mu}_n
		\end{cases}
	\end{equation}
	with $n=1,2,\ldots$ and $\gamma$ the step-size. However to calculate the integration exactly in $g_{\bar{\mu}_{n-1}}$ is not practical, in general the integration is replaced by the empirical summation
	\begin{equation*}
\begin{aligned}
	g_{\bar{\mu}_{n-1}}(\cdot)&=-\int_{\theta\in\Omega}k(\theta,\cdot)\nabla^2\Psi(\theta)^{-1}\nabla_{\theta}\log(\pi)(x)d\mu(\theta)-\int_{\theta\in\Omega}\nabla_{\theta}\cdot\left(k(\theta,\cdot)\nabla^2\Psi(\theta)^{-1}\right)d\mu(\theta)\\
	&\approx 
	\frac{1}{N}\sum_{j=1}^N k(\cdot,\theta_k^j)\nabla^2\Psi(\theta_k^j)^{-1}\nabla_{\theta}\log(\pi)(\theta_k^j)+\nabla_{\theta}\cdot \left(\nabla^2\Psi(\theta)^{-1}k(\cdot,\theta)\right)\mid_{\theta=\theta_k^j}=:{g}_{\hat{\mu}_{n-1}}(\cdot).
\end{aligned}
	\end{equation*}
	So the finite particle Mirrored Stein Variational gradient Descent in Algorithm \ref{alg:1} is equivalent to 
	\begin{equation}
		\begin{cases}
			\hat{\mu}_{n-1}=\nabla\Psi_{\#}\mu_{n-1}\\
			\hat{\mu}_n=\left({\rm I}_d-\gamma g_{\hat{\mu}_{n-1}}\right)_{\#}\hat{\mu}_{n-1}\\
			\mu_n=\nabla\Psi^*_{\#}\hat{\mu}_n
		\end{cases}
	\end{equation}
	with $\mu_0=\frac{1}{N}\sum_{i=1}^N\delta_{\theta_{0}^i}$.
	\section{Assumptions}\label{sec:assump}
	We first state an assumption on the mirror function $\Psi$.
	\begin{assumption}
		\label{asp:mirrormap}
		The mirror function $\Psi(\cdot):\Omega\to (-\infty,+\infty]$ is strongly ${\K}$-convex and the range of the mirror map $\nabla\Psi$ is $\R^d$.
	\end{assumption}
%With the above assumption we can say that distribution $\bar{\pi}=\nabla\Psi_{\#}\pi$ is well defined on the whole space $\R^d$ and $\norm{\nabla^2\Psi^*}_{op}\leq\frac{1}{{\K}}$.
Next, we state the standard assumption on the kernel $k(\cdot,\cdot):\Omega\times\Omega\to\R$.
	\begin{assumption}\label{Al-ker-bound}
		There exists $B_1,B_2>0$ s.t. for all $\theta \in \Omega$,
		$k(\theta,\theta) \leq B_1^2$ and $\partial_{\theta_i}\partial_{\theta'_j}k(\theta,\theta')\mid_{\theta'=\theta}\leq B_2^2$, $\forall i\in[d],j\in[d]$.
	\end{assumption}
	By reproducing property, this is equivalent to say  $\left<k(\theta,\cdot),k(\theta,\cdot)\right>_{\cH_0}\leq B_1^2$ and $ \left<\partial_{\theta_i}k(\theta,\cdot),\partial_{\theta_j}k(\theta,\cdot)\right>_{\cH_0}\leq B_2^2$ for all $\theta\in\Omega$. If $k(\theta,\theta')$ is in the form $f(\theta-\theta')$ and $f$ is smooth at point $0\in\R^d$, then it satisfies \Cref{Al-ker-bound}. One can also normalize $k(\theta,\theta')$ by $k(\theta/{d},\theta'/{d})$ to make $B_2=\cO(\frac{1}{d})$.
	
	\cite{shi2021sampling} assumed $L$-smoothness of $V$, where $V$ is the potential function of the mirrored distribution $\bar{\pi}=\exp(-V)$, that is $$\|\nabla^2V(x)\|_{op}\leq L,$$
	where  $L>0$ is a bounded constant. Different from the $L$-smoothness assumption on $V$, We introduce a weaker smoothness assumption called $(L_0,L_1)$-smoothness from \cite{zhang2019gradient}.
	The assumption reads as in the following:
	\begin{assumption}\label{Al-l0l1}
		$\exists L_0, L_1\geq 0$ s.t. 
		\begin{equation}\label{eq:n089hfd980f}
			\left\| \nabla^2V(x) \right\|_{op} \leq L_0+L_1 \| \nabla V(x) \|,
		\end{equation}
		for any $x\in\mathbb{R}^d$.
	\end{assumption}
	For instance $g(x)=|x|^{2+\delta}$ with $x\in\R$ and $\delta> 0$ satisfies $(L_0,L_1)$-smoothness condition with $L_0=(2+\delta)(1+\delta)^{1+\delta},L_1=1$, while it is not a $L$-smooth function. We also need the a assumption on $\bar{\pi}$.
	%%%%%%%%%%%%%%%%%%%%%Assumption 3%%%%%%%%%%%%%%%%%%%%%%%%%%%%%%
	%We need the following assumption to guarantee \Cref{eq:generalTp} holds.
	\begin{assumption}\label{Tpl-assumption}
		There exists $p\geq 1$, $x_0\in\R^d$ and $s>0$ such that 
		\begin{equation*}
			\int \exp({s\norm{x-x_0}^p})d\bar{\pi}(x)<+\infty.
		\end{equation*}
	\end{assumption}
	If Assumption \ref{Tpl-assumption} holds, then based on results from \cite{bolley2005weighted}~(see \Cref{bolley} in the Appendix), we have
\begin{equation}\label{eq:generalTp}
	W_p(\mu,{\bar{\pi}})\leqslant C_{{\bar{\pi}},p}\left[\KL(\mu \mid {\bar{\pi}})^{\frac{1}{p}}+\left(\frac{\KL(\mu \mid {\bar{\pi}})}{2}\right)^{\frac{1}{2 p}}\right],\quad\forall \mu\in\cP_p(\R^d),
\end{equation}
where
\begin{equation}\label{eq:bolley2}
	C_{{\bar{\pi}},p}:=2 \inf _{x_{0} \in\R^d, s>0}\left(\frac{1}{s}\left(\frac{3}{2}+\log \int \exp({s \norm{x-x_0}^{p}} )d {\bar{\pi}}(x)\right)\right)^{\frac{1}{p}}<+\infty.
\end{equation}
Note constant $C_{{\bar{\pi}},p}$ may depend on dimension, for instance if ${\bar{\pi}}(x)\propto \exp(-\norm{x}^p)$,  we have  $C_{{\bar{\pi}},p}=\cO\left(d^{\frac{1}{p}}\right)$ by simple calculation.
\Cref{Tpl-assumption} is used to guarantee inequality \eqref{eq:generalTp} hold, we can instead directly assume Talagrand's $\text{T}_p$ inequality.
\begin{assumption}\label{asp:Tp}
	For any $\mu\in\cP_p(\R^d)$, we have 
	\begin{equation}\label{eq:Tp}
	W_p(\mu,{\bar{\pi}})\leq \sqrt{\frac{2\KL(\mu\mid{\bar{\pi}})}{\lambda}}.
\end{equation}
\end{assumption}
 It is not clear the condition on $\bar{\pi}$ such that $\text{T}_p$ inequality holds except for a few cases. When $p=1$, a sufficient and necessary condition for \Cref{eq:Tp} to be hold is there exists $x_0\in\R^d, s>0$ such that $	\int \exp({s\norm{x-x_0}^2})d{\bar{\pi}}(x)<+\infty$ ~(see \cite[Theorem 22.10.]{villani2008optimal}). When $V$ is strongly $K$-convex, then \Cref{eq:Tp} holds with $p=2,\lambda=K$~(Bakry-{\'E}mery criterion, see \cite{bakry1985diffusions}), also $\text{T}_2$ is preserved under bounded perturbation~(Holley-Strook criterion, see \cite{holley1986logarithmic} or \cite{steiner2021feynman})\footnote{these two criterions are used for logarithmic Sobolev inequality~(LSI), since LSI implies $\text{T}_2$ with the same constant, they can be also applied on $\text{T}_2$.}.
	
	%%%%%%%%%%%%%%%%%%%%%Assumption 4%%%%%%%%%%%%%%%%%%%%%%%%%%%%%%%
	To bound the Stein Fisher information along trajectory of MSVGD, we also need an assumption on the growth rate of $\nabla V$.
	\begin{assumption}\label{poly-assumption}
		There exists $p\geq 1$ and a constant $C_p$ such that 
		\begin{equation*}
			\big\|\nabla V(x)\big\| \leq  C_p\left(\norm{x}^p+1\right),
		\end{equation*}
		for any $x\in\R^d$.
	\end{assumption}
	
%	In our analysis, the $p$ in \Cref{poly-assumption} can be smaller than the one in \Cref{Tpl-assumption}, however,  for simplicity we always assume they have the same value. 
One simple example that satisfies both  \Cref{Tpl-assumption} and \Cref{poly-assumption} is $V(x)=\norm{x}^p$.
	
	\section{Main Theory}\label{sec:maintheory}
	
	The following proposition is close to the descent lemma, however the step-size $\gamma$ in this proposition depends on the path information. In the following, we always denote $\bar{\pi}=\nabla\Psi_{\#}\pi, \bar{\mu}_n=\nabla\Psi_{\#}\mu_n,\forall n=0,1,2,\ldots$.
	
	\begin{proposition}\label{propl-descent-lemma} Suppose that Assumptions \ref{asp:mirrormap}, \ref{Al-ker-bound} and \ref{Al-l0l1}  are satisfied. Let $\alpha>1$ and choose
		\begin{equation}\label{eq:gamma-cond}
			\gamma \leq \min\left\{\frac{(\alpha-1)K}{\alpha B_2d I_{\Stein}(\mu_n\mid \pi)^{\frac{1}{2}}},\frac{1}{B_1 I_{\Stein}(\mu_n\mid \pi)^{\frac{1}{2}}L_1},\frac{{\K}^2}{\left(\alpha^2B_2^2d^2+{\K}^2(e-1)B_1^2A_n\right)}\right\},
		\end{equation} 
		where $A_n := L_{0}+L_{1}\EE_{\bX\sim\bar{\mu}_n}[\|\nabla V(\bX)\|]$. Then
		\begin{equation}\label{eq:descent-lemma}
			\mathrm{KL}\left(\mu_{n+1} \mid \pi\right)-\mathrm{KL}\left(\mu_{n} \mid \pi\right)
			\leq -\frac{\gamma}{2} I_{\Stein}(\mu_n \mid \pi),
		\end{equation}
		
		%	\peter{$\bX$ is mentioned here, but no commentary whatsoever is made about what this notation means. Also, it's better to write $\mathbb{E}_{\bX \sim \mu_n} \left[\|\nabla V(\bX)\|\right]$, as you do below.}		
	\end{proposition}	
	
	In the above proposition, the upper bound for the step-size $\gamma$ depend on $ I_{\Stein}(\mu_n\mid \pi)^{\frac{1}{2}}$ and $A_n$, these values can be further bounded by $\KL(\mu_n\mid\pi)$ by the following two lemmas.
	\begin{lemma}\label{lem:A3impliesA3}
		If Assumptions \ref{asp:mirrormap} and \ref{Al-ker-bound} are satisfied, then
		\begin{equation}\label{eq:gmun}
		 I_{\Stein}(\mu_n\mid \pi)^{\frac{1}{2} }
			\leq B_1\mathbb{E}_{\bX \sim \bar{\mu}_n}\big[\|\nabla V(\bX)\|\big]+\frac{B_2d}{K}
		\end{equation}
		for all $n\in \NN$.
	\end{lemma}
	Let us denote function $G_p(x):=x^{\frac{1}{p}}+\left(\frac{x}{2}\right)^{\frac{1}{2p}}$ in the following.
	\begin{lemma}\label{lem:exp-grad}
	Suppose Assumptions \ref{Al-l0l1}, \ref{Tpl-assumption} and \ref{poly-assumption} hold. Then,
	\begin{equation}\label{eq:muupper}
		%		\EE_{\bX\sim\mu_n}[\|\nabla V(\bX)\|] \leq 2^{p-1}C_p W_p^p(\bar{\pi},\delta_0)+C_p+2^{p-1}C_p C_{\bar{\pi},p}^pG_p^p(\KL(\mu_n\mid\bar{\pi}))\\
		\EE_{\bX\sim\bar{\mu}_n}[\|\nabla V(\bX)\|] \leq C_p\left(C_{\bar{\pi},p}\left(G_p(\KL({\mu}_n\mid{\pi}))+G_p(\KL({\mu}_0\mid{\pi}))\right)+W_p(\bar{\mu}_0,\delta_0)\right)^p+C_p,
	\end{equation}
		where $C_p$ from \Cref{poly-assumption}, $C_{\bar{\pi},p}$ from \eqref{eq:bolley2}.
	\end{lemma}
	Now, using \Cref{eq:gmun,eq:muupper}  we can give an upper bound for $\gamma$ such that it satisfies condition \eqref{eq:gamma-cond}, then based on \Cref{propl-descent-lemma} we have the following descent lemma which only depends on the initial distribution $\mu_0,\bar{\mu}_0$ and the target distribution $\pi,\bar{\pi}$.
	\begin{theorem}[Descent lemma]
		\label{thm:mainl}
	Suppose Assumptions \ref{asp:mirrormap}, \ref{Al-ker-bound}, \ref{Al-l0l1},
		\ref{Tpl-assumption}, and \ref{poly-assumption} are satisfied.%\aknote{Perhaps adding labels to the assumptions instead numbers, would ease the reading.}
		~	Also, suppose that the step-size $\gamma$ satisfies 
		\begin{equation} \label{eq:gamma-final-condl}\hspace{-10mm}
			\gamma\leq \MM\left(C_p\left(2C_{\bar{\pi},p}G_p(\KL({\mu}_0\mid{\pi}))+W_p(\bar{\mu}_0,\delta_0)\right)^p+C_p\right),
		\end{equation}
		where $\MM(x):=\min\left\{\min\left\{\frac{1}{B_1L_1},\frac{(\alpha-1){\K}}{\alpha B_2d}\right\}\frac{K}{KB_1x+B_2d},\frac{{\K}^2}{\alpha^2B_2^2d^2+{\K}^2B_1^2(e-1)\left(L_1x+L_0\right)}\right\}$,
		then the following bound is true:
		\begin{equation}\label{eq:thm-main}
			\mathrm{KL}\left(\mu_{n+1} \mid \pi\right)-\mathrm{KL}\left(\mu_{n} \mid \pi\right)
			\leq - \frac{\gamma}{2} I_{\Stein}(\mu_n \mid \pi).
		\end{equation}
	\end{theorem}
	Now, based on the above descent lemma, we can derive the complexity bound for MSVGD in population limit.
	\begin{corollary}[Convergence]\label{cor:convergence}
	Suppose Assumptions \ref{asp:mirrormap}, \ref{Al-ker-bound}, \ref{Al-l0l1},
		\ref{Tpl-assumption}, and \ref{poly-assumption} are satisfied.  Also $\gamma$ satisfies condition \eqref{eq:gamma-final-condl} and $B_2=\cO(\frac{1}{d})$, then
		we have
		\begin{equation}
			\frac{1}{n} \sum_{k=0}^{n-1} I_{\Stein}\left(\mu_{k} \mid \pi\right) 
			\leq \frac{2 \KL (\mu_0\mid \pi) }{n \gamma}.
		\end{equation}
		If~ $\bar{\pi}(x)=\exp(-V(x))$ and $\bar{\mu}_0=\cN(0,{\rm I}_d)$,  then to have $\frac{1}{n} \sum_{k=0}^{n-1} I_{\Stein}\left(\mu_{k} \mid \pi\right) \leq\varepsilon$, we need 
		\begin{equation}\label{eq:aaasss}
			n=\tilde{\Omega}\left(
			\frac{8^{p}C_{\bar{\pi},p}^p\Gamma(\frac{p+d+1}{2})^2}{(p+1)^2\Gamma(\frac{d}{2})^2\varepsilon}\right).
		\end{equation}
	\end{corollary}
If $B_2=\cO(1)$, we can replace the kernel $k(\cdot,\cdot)$ by $k(\cdot/d,\cdot/d)$, then it satisfies $B_2=\cO(\frac{1}{d})$. By Stirling formula, we know the dimension dependency of  $\frac{\Gamma(\frac{p+d+1}{2})^2}{\Gamma(\frac{d}{2})^2}$ is of order $1+p$, so if $p$ is not big, there should be $n=\tilde{\Omega}\left(\frac{C_{\bar{\pi},p}^pd^{1+p}}{\varepsilon}\right)$.

\subsection{Better Results under $\text{T}_p$ Inequality when $1\leq p\leq 2$}
	In the next, we assume $\bar{\pi}$ satisfies Talagrand's $\text{T}_p$ inequality \eqref{eq:Tp} instead of \Cref{Tpl-assumption} and keep the other assumptions the same. When $p\in [1,2]$, \Cref{eq:Tp} is more preferable in our analysis since $W_p(\bar{\mu},\bar{\pi})=\cO(\KL(\bar{\mu}\mid\bar{\pi})^{\frac{1}{2}})$ while \Cref{eq:generalTp} gives the bigger bound $W_p(\bar{\mu},\bar{\pi})=\cO(\KL(\bar{\mu}\mid\bar{\pi})^{\frac{1}{p}})$.We focus on the cases when $p\in [1,2]$ and we give a modified version of \Cref{lem:exp-grad}, \Cref{thm:mainl} and \Cref{cor:convergence}.
	\begin{lemma}\label{mlem:exp-grad}
	Suppose Assumptions \ref{Al-l0l1}, \ref{asp:Tp} and \ref{poly-assumption} hold. Then,
	\begin{equation}\label{meq:muupper}
		%		\EE_{\bX\sim\bar{\mu}_n}[\|\nabla V(\bX)\|] \leq 2^{p-1}C_p W_p^p(\bar{\pi},\delta_0)+C_p+2^{p-1}C_p C_{\bar{\pi},p}^pG_p^p(\KL(\bar{\mu}_n\mid\bar{\pi}))\\
		\EE_{\bX\sim\bar{\mu}_n}[\|\nabla V(\bX)\|] \leq C_p\left(\sqrt{\frac{2\KL({\mu}_n\mid{\pi})}{\lambda}}+\sqrt{\frac{2\KL({\mu}_0\mid{\pi})}{\lambda}}+W_p(\bar{\mu}_0,\delta_0)\right)^p+C_p,
	\end{equation}
	where $C_p$ from \Cref{poly-assumption}.
\end{lemma}

	\begin{theorem}
	\label{mthm:mainl}
	Suppose Assumptions \ref{asp:mirrormap}, \ref{Al-ker-bound}, \ref{Al-l0l1},
	and \ref{poly-assumption} are satisfied.%\aknote{Perhaps adding labels to the assumptions instead numbers, would ease the reading.}
	~	Also, suppose that the step-size $\gamma$ satisfies 
	\begin{equation} \label{meq:gamma-final-condl}\hspace{-10mm}
		\gamma\leq\MM\left(C_p\left(\sqrt{\frac{2\KL({\mu}_n\mid{\pi})}{\lambda}}+\sqrt{\frac{2\KL({\mu}_0\mid{\pi})}{\lambda}}+W_p(\bar{\mu}_0,\delta_0)\right)^p+C_p\right),
	\end{equation}
	where  $\MM(x):=\min\left\{\min\left\{\frac{1}{B_1L_1},\frac{(\alpha-1){\K}}{\alpha B_2d}\right\}\frac{K}{KB_1x+B_2d},\frac{{\K}^2}{\alpha^2B_2^2d^2+{\K}^2B_1^2(e-1)\left(L_1x+L_0\right)}\right\}$,
	then the following bound is true:
	\begin{equation}\label{eq:thm-main}
		\mathrm{KL}\left({\mu}_{n+1} \mid {\pi}\right)-\mathrm{KL}\left({\mu}_{n} \mid {\pi}\right)
		\leq - \frac{\gamma}{2} I_{\Stein}({\mu}_n \mid{\pi}).
	\end{equation}
\end{theorem}

	Now, based on the above modified descent lemma, we can derive the complexity bound for MSVGD in population limit.
	\begin{corollary}\label{mcor:convergence}
Suppose Assumptions \ref{asp:mirrormap}, \ref{Al-ker-bound}, \ref{Al-l0l1},\ref{asp:Tp}
	and \ref{poly-assumption} hold. If $\gamma$ satisfies condition \eqref{meq:gamma-final-condl} and $B_2=\cO(\frac{1}{d})$, then
	we have
	\begin{equation}
		\frac{1}{n} \sum_{k=0}^{n-1} I_{\Stein}\left({\mu}_{k} \mid {\pi}\right) 
		\leq \frac{2 \KL ({\mu}_0\mid {\pi}) }{n \gamma}.
	\end{equation}
	If~ $\bar{\pi}(x)=\exp(-V(x))$ and $\bar{\mu}_0=\cN(0,{\rm I}_d)$,  then to make $\frac{1}{n} \sum_{k=0}^{n-1} I_{\Stein}\left({\mu}_{k} \mid {\pi}\right) \leq\varepsilon$, we need 
	\begin{equation}\label{eq:iteration bound}
		n=\tilde{\Omega}\left(\frac{d^{\frac{(p+2)(p+1)}{4}}}{\lambda^{\frac{p}{2}}\varepsilon}\right).
	\end{equation}
\end{corollary}

	When $1\leq p\leq 2$, if we assume $\pi$ satisfies Talagrand's $\text{T}_p$ inequality \eqref{eq:Tp} instead of 
	\Cref{Tpl-assumption}, $n$ will have a slightly better dimension dependency, that is of order $\frac{(p+2)(p+1)}{4}$.

	%In the above corollary, we assume $\pi(x)=\exp(-V(x))$ instead of $\pi(x)\propto \exp(-V(x))$ is purely out of we can have  an upper bound for $\KL(\mu_0\mid\pi)$ with $\mu_0=\cN(0,{\rm I}_d)$ in this case~(see \Cref{lem:upperboundmu0} in the appendix). Term
	%$\left(\frac{p}{e}\right)^{\frac{p-1}{2}}$ is from the estimation of $\int_{x\in\R}|x|^{p+1}\exp(-|x|^2/2)dx$ when we upper bound $\KL(\mu_0\mid\pi)$.
	%If we assume Talagrand's $\text{T}_p$ inequality holds instead of \Cref{Tpl-assumption}, then we will have dimension dependency of order $\frac{(p+2)(p+1)}{4}$ instead of $p+1$ in the above corollary. If we assume $(L,0)$-smoothness of $V$, $\pi$ satisfies $\text{T}_1$ inequality and $\mu_0=\cN(x^{\star},\frac{1}{L}{\rm I}_d)$ where $x^{\star}$ is the global minimum of $V$ instead of $\cN(0,{\rm I}_d)$, then the general scheme of our analysis will provide $n={\Omega}(\frac{Ld^{3/2}}{\lambda^{1/2}\varepsilon})$ which cover the result in \cite{salim2021complexity}, see \Cref{rmk:33}  in the appendix. 

	\begin{remark}
		Here we provide a sufficient condition on which $\lim_{n\to\infty}I_{\Stein}(\mu_n\mid\pi)$ implies $\mu_n\to\pi$ weakly, this condition can be found in \cite{gorham2017measuring}. Since Mirrored Stein Fisher information  depends on the target distribution $\bar{\pi}$, mirror function $\Psi$ and kernel $k(\cdot,\cdot)$, we need the following two properties on $\bar{\pi}$ and $k(\nabla\Psi^*(\cdot),\nabla\Psi^*(\cdot))$ respectively:
		
		\begin{enumerate}
			\item $\bar{\pi}$ is distant dissipative, that is $\kappa_{0} \triangleq \liminf _{r \rightarrow \infty} \kappa(r)>0$  with
			\begin{equation*}
				\kappa(r)=\inf \left\{2 \frac{\langle \nabla V(x)-\nabla V(y), x-y\rangle}{\|x-y\|^{2}}:\|x-y\|=r\right\}.
			\end{equation*}
			If $V$ is strongly convex outside a compact set, then $\bar{\pi}$ is distant dissipative, for instance $V(x)=\norm{x}^{2+\delta}$ with $\delta\geq 0$.
			\item $k(\nabla\Psi^*(\cdot),\nabla\Psi^*(\cdot))$ is an inverse multiquadratic kernel, i.e., $k(\nabla\Psi^*(x), \nabla\Psi^*(y))=\left(c^{2}+\|x-y\|^{2}\right)^{\beta}$ for some $c>0$ and $\beta \in(-1,0)$. It is easy to check that \Cref{Al-ker-bound} is satisfied.
		\end{enumerate}
	\end{remark}

	\section{Conclusion}\label{sec:conclusion}
	
	We proved a descent lemma for minimizing the $\KL$ divergence with the MSVGD algorithm under $(L_0,L_1)$-smoothness condition.  Though in this paper we assume $k(\cdot,\cdot)$ is scalar, our analysis can also be applied on matrix kernel.
	The result implies $\cO(1/n)$ convergence under Mirrored Stein Fisher information. The results in this note match the ones from \cite{sun2022convergence}, this is not surprising since MSVGD is equivalent to SVGD for target distribution $\nabla\Psi_{\#}\pi$.
	The proved results remain on the theoretical level, the algorithm we analyzed in the infinite particle regime which is not implementable on a computer. There is paper analysing the finite particle regime \cite{shi2021sampling}, however they assumed a $L$-smooth $V$ and their result is asymptotic. A relevant analysis of the finite particle
	regime still lacks certain clarity and is yet to be developed.

	\clearpage
	\bibliographystyle{plainnat}
	\bibliography{bibliography,math1}
	\clearpage
	\appendix
	
	\tableofcontents
	\clearpage
	\part*{Appendix}\label{appendix}
	For simplicity, in this appendix we will denote $k_{\Psi}(\cdot,\cdot):=k(\nabla\Psi^*(\cdot),\nabla\Psi^*(\cdot))$.
	Since  $\KL(\nabla\Psi_{\#}\mu\mid \nabla\Psi_{\#}\pi)=\KL(\mu\mid\pi)$ for diffeomorphism $\nabla\Psi$, so it is enough to prove the results using  $\KL(\bar{\mu}\mid\bar{\pi})$, where $\bar{\mu}=\nabla\Psi_{\#}\mu,\bar{\pi}= \nabla\Psi_{\#}\pi$.

	\section{Proof of \Cref{propl-descent-lemma}}\label{subsec:1}
	We follow the proof procedure in \cite{sun2022convergence}. First due to \Cref{lem:oonnee}, we have $\norm{g_{\bar{\mu}_n}}_{\cH}^2=I_{\Stein}(\mu_n\mid\pi)$.  For the sake of brevity we will omit the index when referring to $g_{{\bar{\mu}}_n}$ and will write simply $g$. 
	Let us define $\phi_{t} := I - t g$ and  $\bar{\rho}_t:= {\phi_{t}}_{\#} {\bar{\mu}}_n$ for every $t \in[0, \gamma]$.  
	Then, applying Taylor formula to the function \begin{equation}\label{eq:g9y980fd8}\varphi(t):=\mathrm{KL}\left(\bar{\rho}_{t} \mid \bar{\pi}\right)\end{equation} we have the following
	\begin{equation}\label{eq:phi-taylor}
		\varphi(\gamma)=\varphi(0)+\gamma \varphi^{\prime}(0)+\int_{0}^{\gamma}(\gamma-t) \varphi^{\prime \prime}(t) {\rm d} t.
	\end{equation}
	By the definition of the MSVGD iteration, we have that $\varphi(0)={\KL}\left({\bar{\mu}}_{n} \mid \bar{\pi}\right)$ and $\varphi(\gamma)={\KL}\left({\bar{\mu}}_{n+1} \mid \bar{\pi}\right)$. Let us now compute the term of \eqref{eq:phi-taylor} corresponding to the first order derivative.{ In order to do that we show that $\phi_t$ is a diffeomorphism.}
	\begin{lemma}\label{lem:diff}
		Suppose that Assumption \ref{Al-ker-bound} holds. Then, for any $x \in \R^d$ and  $h \in \mathcal{H}$, 
		\begin{equation*} 
			\|J h(x)\|_{{\rm HS}} \leq \frac{B_2d}{{\K}}\|h\|_{\mathcal{H}}.
		\end{equation*}
	\end{lemma}
	Applying the lemma to the function $g$, we obtain the following:
	\begin{equation}\label{eq:ineq-op-HS}
		\|t J g(x)\|_{op} \leq \|t J g(x)\|_{{\rm HS}} \leq t \frac{B_2d}{{\K}}\|g\|_{\mathcal{H}} < 1.
	\end{equation}
	The latter inequality is due to the condition on the step-size $\gamma$~($\gamma\leq\frac{(\alpha-1){\K}}{\alpha B_2d\norm{g}_{\cH}}$). The inequality \eqref{eq:ineq-op-HS} implies
	that $\phi_{t}$ is a diffeomorphism. 
	Therefore, $\bar{\rho}_{t}$ admits a density given by the formula of transformation of probability densities:
	\begin{equation*}
		\bar{\rho}_{t}(x)=\left|J \phi_{t}\left(\phi_{t}^{-1}(x)\right)\right|^{-1} {\bar{\mu}}_{n}\left(\phi_{t}^{-1}(x)\right).
	\end{equation*}
	Changing the variable of integration and applying the transfer lemma we get the following formula for $\varphi(t)$:
	\begin{equation*}
		\begin{aligned}
			\varphi(t) &=\int \log \left(\frac{\bar{\rho}_{t}(y)}{\bar{\pi}(y)}\right) \bar{\rho}_{t}(y) \rmd y \\
			&=\int \log \left(\frac{{\bar{\mu}}_{n}(x)\left|J \phi_{t}(x)\right|^{-1}}{\bar{\pi}\left(\phi_{t}(x)\right)}\right) {\bar{\mu}}_{n}(x) \rmd x\\
			&=\int  \left[ \log \left({\bar{\mu}}_{n}(x)\right) + \log \left(\left|J \phi_{t}(x)\right|^{-1}\right) - \log \left({\bar{\pi}\left(\phi_{t}(x)\right)}\right)
			\right]  {\bar{\mu}}_{n}(x) \rmd x.
		\end{aligned}
	\end{equation*}
	Let us then compute the time derivative of $\varphi(t)$. Taking the derivative inside and applying Jacobi's formula for matrix determinant differentiation we obtain the following equality: 
	\begin{equation*}
		\varphi^{\prime}(t)=-\int \operatorname{tr}\left(J \phi_{t}(x)^{-1} \frac{d J \phi_{t}(x)}{d t}\right) {\bar{\mu}}_{n}(x) 
		\rmd x
		- \int\left\langle\nabla \log \bar{\pi}\left(\phi_{t}(x)\right), \frac{d \phi_{t}(x)}{d t}\right\rangle {\bar{\mu}}_{n}(x) \rmd x.
	\end{equation*}
	By definition, $d\phi_t / dt = g$. Therefore, we can use the explicit expression of $\phi_{t}$ to write:
	\begin{equation*}
		\varphi^{\prime}(t)=\int \operatorname{tr}\left(J \phi_{t}(x)^{-1} J g(x)\right) {\bar{\mu}}_{n}(x) \rmd x
		+ \int \left \langle \nabla V\left(\phi_{t}(x)\right), g(x)\right\rangle {\bar{\mu}}_{n}(x) \rmd x .
	\end{equation*}
	The Jacobian at time $t=0$ is simply equal to the identity matrix since $\phi_{0}={\rm I}_d$.   It follows that $\operatorname{tr}\left(J \phi_{0}(x)^{-1} J g(x)\right)=$ $\operatorname{tr}(J g(x))=\operatorname{div}(g)(x)$ by the definition of the divergence operator. Using integration by parts:
	\begin{equation}\label{eq:phi'(0)}
		\begin{aligned}
			\varphi^{\prime}(0) &=-\int[-\operatorname{div}(g)(x)-\langle\nabla \log \bar{\pi}(x), g(x)\rangle] {\bar{\mu}}_{n}(x) \rmd x \\
			&=-\int\left\langle\nabla \log \left(\frac{{\bar{\mu}}_{n}}{\bar{\pi}}\right)(x), g(x)\right\rangle {\bar{\mu}}_{n}(x) \rmd x\\
			&= -\left\langle P_{\bar{\mu}}\nabla \log \left(\frac{{\bar{\mu}}_{n}}{\bar{\pi}}\right)(x), g(x) \right\rangle_{\cH}\\
			&= - \|g\|_{\cH}^2 .
		\end{aligned}
	\end{equation}
	Next, we calculate the term of \eqref{eq:phi-taylor} that contains the second derivative. First,
	\begin{equation*}
		\varphi^{\prime \prime}(t)
		=\int\left[{\operatorname{tr}\left(\left(J g(x)\left(J \phi_{t}(x)\right)^{-1}\right)^{2}\right)}
		+\left\langle g(x), \nabla^2 V\left(\phi_{t}(x)\right) g(x)\right\rangle\right] {\bar{\mu}}_{n}(x) \rmd x.
	\end{equation*}
	From the definition of the function $\phi_t$, we know that $J \phi_{t}(x) = ({\rm I}_d + t Jg)(x)$. 
	Thus, $J \phi_{t}^{-1}$ and $Jg$ commute and 
	$\operatorname{tr}\left(\left(J g(x)\left(J \phi_{t}(x)\right)^{-1}\right)^{2}\right)
	=\left\|J g(x)\left(J \phi_{t}(x)\right)^{-1}\right\|_{\rm H S}^{2}$. 
	
	Overall we have the following:
	\begin{equation*}
		\varphi^{\prime \prime}(t)
		=\underbrace{\int \left\|J g(x)\left(J \phi_{t}(x)\right)^{-1}\right\|_{\rm H S}^{2} {\bar{\mu}}_{n}(x) \rmd x}_
		{:=\psi_1(t)}
		+\underbrace{\int\left\langle g(x), H_{V}\left(\phi_{t}(x)\right) g(x)\right\rangle{\bar{\mu}}_{n}(x) \rmd x}_{:=\psi_2(t)} .
	\end{equation*}
	First, we bound $\psi_{1}(t)$. Cauchy-Schwarz implies that
	\begin{equation*}
		\left\|J g(x)\left(J \phi_{t}(x)\right)^{-1}\right\|_{\rm H S}^{2} 
		\leq \|J g(x)\|_{\rm H S}^{2}\left\|\left(J \phi_{t}(x)\right)^{-1}\right\|_{o p}^{2}.
	\end{equation*}
	From \Cref{lem:diff}, we have  $\|J g(x)\|_{\rm H S} \leq \frac{B_2d}{{\K}}\|g\|_{\cH}$. 
	To bound the second term, let us recall that $\phi_t = I - tg$ and that $t \leq \gamma$. Thus, the following bound is true:
	\begin{equation*}
		\begin{aligned}
			\left\|\left(J \phi_{t}(x)\right)^{-1}\right\|_{o p} 
			& = \left\|\big( ( {\rm I}_d - t J g)(x)\big)^{-1}\right\|_{o p} \leq \sum_{k=0}^{\infty}\|t J g(x)\|_{o p}^{k} \leq \sum_{k=0}^{\infty}\|\gamma J g(x)\|_{\rm H S}^{k}.
		\end{aligned}
	\end{equation*}
	Recalling \eqref{eq:gamma-cond} and combining it with \Cref{lem:diff} we obtain 
	\begin{equation*}
		\begin{aligned}
			\left\|\left(J \phi_{t}(x)\right)^{-1}\right\|_{o p} \leq \sum_{k=0}^{\infty}(\gamma \frac{B_2d}{{\K}} \|g\|_{\cH})^{k} \leq \sum_{k=0}^{\infty} \left(\frac{\alpha-1}{\alpha}\right)^k  = \alpha.
		\end{aligned}
	\end{equation*}
	
	Summing up, we have that
	\begin{equation}\label{eq:psi1-bound}
		\psi_1(t) \leq \frac{\alpha^2 B_2^2d^2}{{\K}^2}\|g\|_{\cH}^2.
	\end{equation}
	
	Next, we bound $\psi_{2}$. By definition
	\begin{align*}
		\psi_2(t) &= \EE_{\bX\sim{\bar{\mu}}_n} \left[\left\langle g(\bX), \nabla^2 V\left(\phi_{t}(\bX)\right) g(\bX)\right\rangle \right] \leq \EE_{\bX\sim{\bar{\mu}}_n} \left[ \|\nabla^2{V}\left(\phi_t(\bX)\right)\|_{op} \|g(\bX)\|_2^2 \right].
	\end{align*}
	Let us bound the norm of $g(x)$. The reproduction property of the RKHS yields the following:
	\begin{equation}\label{eq:norm2-H}
		\|g(x)\|_2^{2}=\sum_{i=1}^{d}\left\langle k(x, .), g_{i}\right\rangle_{\mathcal{H}_{0}}^{2} \leq\|k(x, .)\|_{\mathcal{H}_{0}}^{2}\|g\|_{\mathcal{H}}^{2} \leq B_1^{2}\|g\|_{\mathcal{H}}^{2}.
	\end{equation}
	Therefore,
	\begin{equation}\label{eq:psi2-initial-bound}
		\psi_2(t) \leq  B_1^{2}\|g\|_{\mathcal{H}}^{2}\EE_{\bX\sim{\bar{\mu}}_n} \left[ \|\nabla^2{V}\left(\phi_t(\bX)\right)\|_{op} \right].
	\end{equation}
	Let us bound $\EE_{\bX\sim{\bar{\mu}}_n} \left[ \|H_{V}\left(\phi_t(\bX)\right)\|_{op} \right]$. \Cref{Al-l0l1} implies the following inequality:
	\begin{equation*}
		\left\| \nabla^2V (\phi_t(x))\right\|_{op} \leq L_{0}+L_{1}\|\nabla V(\phi_t(x))\|,
	\end{equation*}
	for every $x \in \RR^d$.
	
	To bound the term $\|\nabla V(\phi_t(x))\|$, we introduce the following lemma from \cite{sun2022convergence} without proof.
	\begin{lemma}\label{lem:grad-v}
		Let $V$ be an $\left(L_{0}, L_{1}\right)$-smooth function and $\Delta>0$ be a constant. 
		For any $x,x^{+} \in \RR^d$ such that $\left\|x^{+}-x\right\| \leq \Delta $, we have 
		\begin{equation*}
			\left\|\nabla V\left(x^{+}\right)\right\| \leq 
			\frac{L_0 }{ L_1 } \left(\exp(\Delta L_1) - 1\right)  +  \|\nabla V(x)\|\exp(\Delta L_1).
		\end{equation*}
	\end{lemma} 
	We will apply \Cref{lem:grad-v} to $\phi_t(x)$ and $\phi_0(x)$. By definition, $\phi_t(x) - \phi_0(x) = t g(x)$ and
	according to inequality \eqref{eq:norm2-H},
	\begin{equation*}
		\|\phi_t(x) - \phi_0(x)\|\leq t B_1\| g\|_{\cH}.
	\end{equation*}
	Thus, using \Cref{lem:grad-v} for $x = \phi_0(x) $, $x^+ = \phi_t(x)$ and $\Delta = tB_1\|g\|_{\cH}$, we obtain the following:
	\begin{equation}\label{eq:bound-Hess-phi}
		\begin{aligned}
			\left\| \nabla^2V (\phi_t(x))\right\|_{op} 
			&\leq  L_0 + L_1 \left(\frac{L_0}{L_1}\big(\exp(t B_1 \|g\|_{\cH} L_1) - 1\big) 
			+ \left\|\nabla V\left(\phi_0(x) \right)\right\| \exp(t B_1 \|g\|_{\cH}L_1)\right) \\
			&= \big(L_{0} + L_1\left\|\nabla V\left(x\right)\right\|\big)\exp (t B_1 \|g\|_{\cH} L_1),
		\end{aligned}
	\end{equation}
	where the last equality is due to $\phi_0(x) = x$.
	Combining \eqref{eq:psi2-initial-bound} and \eqref{eq:bound-Hess-phi} we obtain
	\begin{equation*}\label{eq:psi2-final-bound}
		\psi_2(t)\leq   B_1^2 \|g\|_{\cH}^2 \left(L_{0} + L_1 \EE_{\bX\sim{\bar{\mu}}_n}\big[\left\|\nabla V (\bX) \right\|\big]\right) 
		\exp (t B_1 \|g\|_{\cH} L_1).
	\end{equation*}
	Summing up, the bounds on $\psi_1$ and $\psi_2$ yield the following inequality:
	\begin{equation*}
		\begin{aligned}
			\varphi^{\prime \prime}(t)& \leq  \|g\|_{\cH}^2 \Big[ \frac{\alpha^2B_2^2d^2}{{\K}^2} 
			+ B_1^2\Big(L_{0} + L_1 \EE_{\bX\sim{\bar{\mu}}_n}\big[\left\|\nabla V (\bX) \right\|\big]\Big) 
			\exp (t B_1 \|g\|_{\cH} L_1)\Big]\\
			&=  \|g\|_{\cH}^2 \Big[\frac{ \alpha^2B_2^2 d^2}{{\K}^2}
			+ B_1^2A_n
			\exp (t B_1 \|g\|_{\cH} L_1)\Big],
		\end{aligned}
	\end{equation*}
	recall that by definition $A_n = \big(L_{0} + L_1 \EE_{\bX\sim{\bar{\mu}}_n}\big[\left\|\nabla V (\bX) \right\|\big]\big)$. Use the previous upper bound, we have 
	\begin{equation*}
		\begin{aligned}
			\int_{0}^{\gamma}(\gamma-t)\varphi^{\prime\prime}(t)dt&\leq \|g\|_{\cH}^2 \int_{0}^{\gamma}(\gamma-t)\Big[ \frac{\alpha^2B_2^2d^2}{{\K}^2}
			+B_1^2A_n
			\exp (t B_1 \|g\|_{\cH}L_1 )\Big]dt\\
			&=\frac{\gamma^2}{2} \|g\|_{\cH}^2\frac{ \alpha^2 B_2^2d^2}{{\K}^2}
			+ B_1^2 \|g\|_{\cH}^2A_n \frac{\exp (\gamma B_1 \|g\|_{\cH}L_1 )-\gamma _1B\|g\|_{\cH}L_1-1}{B_1^2\|g\|^2_{\cH}L_1^2}.
		\end{aligned}
	\end{equation*}
	One can check that $\exp(t) - t - 1 \leq (e-1)t^2/2$, when $t \in [0,1]$, since $\gamma B_1 \|g\|_{\cH}L_1 < 1$, we deduce that 
	\begin{equation}
		\int_{0}^{\gamma}(\gamma-t)\varphi^{\prime\prime}(t)dt\leq \|g\|^2_{\cH}\left(\frac{\gamma^2\alpha^2B_2^2d^2}{2{\K}^2}+\frac{e-1}{2}B_1^2A_n\gamma^2\right),
	\end{equation}
	combine with \eqref{eq:phi'(0)}, we finally have 
	\begin{equation*}
		\begin{aligned}
			\varphi(\gamma) - \varphi(0) &\leq - \gamma \|g\|_{\cH}^2
			+\gamma^2\|g\|^2_{\cH}\left(\frac{\alpha^2B_2^2d^2}{2{\K}^2}+\frac{e-1}{2}B_1^2A_n\right)\\
			&=-\gamma\left(1-\frac{1}{2}\gamma \left(\frac{\alpha^2B_2^2d^2}{{\K}^2}+(e-1)B_1^2A_n\right)\right)\|g\|_{\cH}^2 \\
			&\overset{(i)}{\leq}-\frac{\gamma}{2}\|g\|_{\cH}^2,
		\end{aligned}
	\end{equation*}
	$(i)$ is since we choose $\gamma\leq\frac{{\K}^2}{\left(\alpha^2B_2^2d^2+{\K}^2(e-1)B_1^2A_n\right)}$.
	
	Finally, by \Cref{lem:oonnee} $\|g\|_{\cH}^2 = I_{\Stein}({{\mu}}_n \mid {\pi})$.
	This concludes the proof of the \Cref{propl-descent-lemma}.
	\begin{remark}
		One may notice that we did not use the exact formula of $g$ except for \eqref{eq:phi'(0)}. In fact the proposition remains true for a general $g \in \cH$ with a slight change. If we skip the last equality in \eqref{eq:phi'(0)} and repeat the other steps, then we get the following:
		\begin{equation}
			\varphi(\gamma) - \varphi(0) \leq  
			- \gamma \left\langle P_{\bar{\mu}}\nabla \log \left(\frac{{\bar{\mu}}_{n}}{\bar{\pi}}\right)(x), g(x) \right\rangle_{\cH}
			+\frac{\gamma^2}{2}\|g\|^2_{\cH}\left(\frac{\alpha^2B_2^2d^2}{{\K}^2}+{(e-1)B_1^2A_n}\right).
		\end{equation}
	\end{remark}

	\section{ Lemmas}
	
	The following lemma is from \cite[Corollary 2.3.]{bolley2005weighted}.
	\begin{lemma}\label{bolley}
		Let $\mathcal{X}$ be a measurable space equipped with a measurable distance $\bf d$, let $p \geqslant 1$ and let $\nu$ be a probability measure on $\mathcal{X}$. Assume that there exist $x_{0} \in \mathcal{X}$ and $s>0$ such that $\int \exp({s {\bf d}\left(x_{0}, x\right)^{p}}) d \nu(x)$ is finite. Then
		$$
		\forall \mu \in \cP(\mathcal{X}), \quad W_{p}(\mu, \nu) \leqslant C_{\nu,p}\left[\KL(\mu \mid \nu)^{\frac{1}{p}}+\left(\frac{\KL(\mu \mid \nu)}{2}\right)^{\frac{1}{2 p}}\right],
		$$
		where
		\begin{equation}\label{eq:bolley1}
			C_{\nu,p}:=2 \inf _{x_{0} \in X, s>0}\left(\frac{1}{s}\left(\frac{3}{2}+\log \int \exp({s {\bf d}\left(x_{0}, x\right)^{p}}) d \nu(x)\right)\right)^{\frac{1}{p}}<+\infty.
		\end{equation}
	\end{lemma}
	
%	The following lemma is an integral form of Gr{\"o}nwall inequality from \cite[Chapter II.]{amann2011ordinary}, which is used in the proof of \Cref{lem:grad-v}.
	
	%	\begin{lemma}[Gr{\"o}nwall Inequality]\label{lem:gronwalllll}
		%		Assume $\phi, B:[0, T] \rightarrow \mathbb{R}$ are bounded non-negative measurable function and $C:[0, T] \rightarrow \mathbb{R}$ is a non-negative integrable function with the property that
		%		\begin{equation}
			%			\label{eq:grrrr1}
			%			\phi(t) \leq B(t)+\int_{0}^{t} C(\tau) \phi(\tau) d \tau \quad \text { for all } t \in[0, T]
			%		\end{equation}
		%		Then
		%		\begin{equation}
			%			\label{eq:grrrr2}
			%			\phi(t) \leq B(t)+\int_{0}^{t} B(s) C(s) \exp \left(\int_{s}^{t} C(\tau) d \tau\right) d s \quad \text { for all } t \in[0, T].
			%		\end{equation}
		%	\end{lemma}
	
	\subsection{Proof of Lemmas}
		\begin{proof}[\textbf{Proof of \Cref{lem:oonnee}}]
		
		By the formula of transformation of probability densities, we first have \newline$\bar{\mu}(x)=\mu(\nabla\Psi^*(x))|\det\nabla^2\Psi(\nabla\Psi^*(x))|^{-1}$ and $\bar{\pi}(x)=\pi(\nabla\Psi^*(x))
		|\det\nabla^2\Psi(\nabla\Psi^*(x))|^{-1}$, so 
		\begin{equation}
			\log(\frac{\mu}{\pi})(\theta)=\log(\frac{\bar{\mu}}{\bar{\pi}})(x),
		\end{equation}
		where $x=\nabla\Psi(\theta)$. Since $\mu(\theta)=\left(\nabla\Psi^*_{\#}\bar{\mu}\right)(\theta)$, we have 
		\begin{equation}
			\begin{aligned}
				&\int_{\theta\in\Omega}k(\theta,\cdot)\nabla^2\Psi(\theta)^{-1}\nabla_{\theta}\log(\frac{\mu}{\pi})(\theta)d\mu(\theta)\\
				&=\int_{\theta\in\Omega}k(\theta,\cdot)\nabla^2\Psi(\theta)^{-1}\nabla_{\theta}\log(\frac{\mu}{\pi})(\theta)d\left(\nabla\Psi^*_{\#}\bar{\mu}\right)(\theta)\\
				&=\int_{x\in\R^d}k(\nabla\Psi^*(x),\cdot)\nabla^2\Psi(\nabla\Psi^*(x))^{-1}\nabla_{\theta}\log(\frac{\mu}{\pi})(\nabla\Psi^*(x))d\bar{\mu}(x)\\
				&=\int_{x\in\R^d}k(\nabla\Psi^*(x),\cdot)\nabla^2\Psi(\nabla\Psi^*(x))^{-1}\nabla_{\theta}\log(\frac{\bar{\mu}}{\bar{\pi}})(x)d\bar{\mu}(x)\\
				&=\int_{x\in\R^d}k(\nabla\Psi^*(x),\cdot)\nabla^2\Psi(\nabla\Psi^*(x))^{-1}\nabla^2\Psi(\nabla\Psi^*(x))\nabla_{x}\log(\frac{\bar{\mu}}{\bar{\pi}})(x)d\bar{\mu}(x)\\
				&=\int_{x\in\R^d}k(\nabla\Psi^*(x),\cdot)\nabla_{x}\log(\frac{\bar{\mu}}{\bar{\pi}})(x)d\bar{\mu}(x).
			\end{aligned}
		\end{equation}
		Similarly, we have
		\begin{equation}
			\begin{aligned}
				&\int_{\theta\in\Omega}\int_{\theta'\in\Omega}k(\theta,\theta')\inner{\nabla^2\Psi(\theta)^{-1}\nabla_{\theta}\log(\frac{\mu}{\pi})(\theta)}{\nabla^2\Psi(\theta')^{-1}\nabla_{\theta'}\log(\frac{\mu}{\pi})(\theta')}d\mu(\theta)d\mu(\theta')\\
				&=\int_{\theta\in\Omega}\int_{\theta'\in\Omega}k(\theta,\theta')\inner{\nabla^2\Psi(\theta)^{-1}\nabla_{\theta}\log(\frac{\mu}{\pi})(\theta)}{\nabla^2\Psi(\theta')^{-1}\nabla_{\theta'}\log(\frac{\mu}{\pi})(\theta')}d\left(\nabla\Psi^*_{\#}\bar{\mu}\right)(\theta)d\left(\nabla\Psi^*_{\#}\bar{\mu}\right)(\theta')\\
				&=\int_{x\in\R^d}\int_{y\in\R^d}k(\nabla\Psi^*(x),\nabla\Psi^*(y))\\
				&\qquad\inner{\nabla^2\Psi(\nabla\Psi^*(x))^{-1}\nabla_{\theta}\log(\frac{\mu}{\pi})(\nabla\Psi^*(x))}{\nabla^2\Psi(\nabla\Psi^*(y))^{-1}\nabla_{\theta'}\log(\frac{\mu}{\pi})(\nabla\Psi^*(y))}d\bar{\mu}(x)d\bar{\mu}(y)\\
				&=\int_{x\in\R^d}\int_{y\in\R^d}k(\nabla\Psi^*(x),\nabla\Psi^*(y))\inner{\nabla^2\Psi(\nabla\Psi^*(x))^{-1}\nabla_{\theta}\log(\frac{\bar{\mu}}{\bar{\pi}})(x)}{\nabla^2\Psi(\nabla\Psi^*(y))^{-1}\nabla_{\theta'}\log(\frac{\bar{\mu}}{\bar{\pi}})(y)}d\bar{\mu}(x)d\bar{\mu}(y)\\
				&=\int_{x\in\R^d}\int_{y\in\R^d}k(\nabla\Psi^*(x),\nabla\Psi^*(y))\\
				&\qquad\inner{
					\nabla^2\Psi(\nabla\Psi^*(x))^{-1}\nabla^2\Psi(\nabla\Psi^*(x))\nabla_{x}\log(\frac{\bar{\mu}}{\bar{\pi}})(x)}{
					\nabla^2\Psi(\nabla\Psi^*(y))^{-1}\nabla^2\Psi(\nabla\Psi^*(y))\nabla_{y}\log(\frac{\bar{\mu}}{\bar{\pi}})(y)}
				d\bar{\mu}(x)d\bar{\mu}(y)\\
				&=\int_{x\in\R^d}\int_{y\in\R^d}k(\nabla\Psi^*(x),\nabla\Psi^*(y))\inner{\nabla_{x}\log(\frac{\bar{\mu}}{\bar{\pi}})(x)}{\nabla_{y}\log(\frac{\bar{\mu}}{\bar{\pi}})(y)}d\bar{\mu}(x)d\bar{\mu}(y)\\
				&=\norm{g_{\bar{\mu}}}_{\cH}^2,
			\end{aligned}
		\end{equation}
		where $x=\nabla\Psi(\theta),y=\nabla\Psi(\theta')$.
	\end{proof}
	
	\begin{proof}[\textbf{Proof of \Cref{lem:ttwwoo}}]
		The proof is direct.
		\begin{equation}
			\begin{aligned}
				&\int_{\theta\in\Omega}k(\theta,\cdot)\nabla^2\Psi(\theta)^{-1}\nabla_{\theta}\log(\frac{\mu}{\pi})(\theta)d\mu(\theta)\\
				&=\int_{\theta\in\Omega}k(\theta,\cdot)\nabla^2\Psi(\theta)^{-1}\nabla_{\theta}\log(\mu)(\theta)\mu(\theta)d\theta-
				\int_{\theta\in\Omega}k(\theta,\cdot)\nabla^2\Psi(\theta)^{-1}\nabla_{\theta}\log(\pi)(\theta)\mu(\theta)d\theta\\
				&=\int_{\theta\in\Omega}k(\theta,\cdot)\nabla^2\Psi(\theta)^{-1}\nabla_{\theta}\mu(\theta)d\theta-
				\int_{\theta\in\Omega}k(\theta,\cdot)\nabla^2\Psi(\theta)^{-1}\nabla_{\theta}\log(\pi)(\theta)\mu(\theta)d\theta\\
				&\overset{(i)}{=}\int_{\theta\in\Omega}\nabla_{\theta}\cdot\left(k(\theta,\cdot)\nabla^2\Psi(\theta)^{-1}\right)d\mu(\theta)-
				\int_{\theta\in\Omega}k(\theta,\cdot)\nabla^2\Psi(\theta)^{-1}\nabla_{\theta}\log(\pi)(\theta)\mu(\theta)d\theta,
			\end{aligned}
		\end{equation}
		$(i)$ is due to integration by parts and the assumption that $k(\theta,\cdot)\nabla^2\Psi(\theta)^{-1}\mu(\theta)\to 0$ as $\theta\to\partial\Omega$.
	\end{proof}
	
	Before we proceed to prove the left lemmas, we need to do some calculation first.
		Let $e_i$ be the $i$-th standard basis of $\R^d$, then
	\begin{equation*}
		\begin{aligned}
			&\langle \partial_{x_{i}}k(\nabla\Psi^*(x),\cdot),\partial_{y_{i}}k(\nabla\Psi^*(y),\cdot)\rangle_{\cH_0}\\
			&=\lim_{\epsilon_1\to 0}\lim_{\epsilon_2\to 0}\langle \frac{k(\nabla\Psi^*(x+\epsilon_1e_i),\cdot)-k(\nabla\Psi(x),\cdot)}{\epsilon_1}, \frac{k(\nabla\Psi^*(y+\epsilon_2e_i),\cdot)-k(\nabla\Psi^*(y),\cdot)}{\epsilon_2}\rangle_{\cH_0}\\
			&=\lim_{\epsilon_1\to 0}\frac{1}{\epsilon_1}\left\{\lim_{\epsilon_2\to 0}\frac{k(\nabla\Psi^*(x+\epsilon_1e_i),\nabla\Psi^*(y+\epsilon_2e_i))-k(\nabla\Psi^*(x+\epsilon_1e_i),\nabla\Psi^*(y))}{\epsilon_1\epsilon_2}\right.\\
			&\left.\quad-\lim_{\epsilon_2\to 0}\frac{k(\nabla\Psi^*(x),\nabla\Psi^*(y+\epsilon_2e_i))-k(\nabla\Psi^*(x),\nabla\Psi^*(y))}{\epsilon_1\epsilon_2}\right\}\\
			&=\lim_{\epsilon_1\to 0}\frac{\sum_{j=1}^d\partial_{i}\partial_j\Psi^*(y)\left(\partial_{\theta'_j}k(\nabla\Psi^*(x+\epsilon_1e_i),\theta')-\partial_{\theta'_j}k(\nabla\Psi^*(x),\theta')\right)}{\epsilon_1}\mid_{\theta'=\nabla\Psi^*(y)}\\
			&=\sum_{j,j'}\partial_i\partial_j\Psi^*(x)\partial_i\partial_{j'}\Psi^*(y)\partial_{\theta_{j}}\partial_{\theta'_{j'}}k(\theta,\theta')\mid_{\theta=\nabla \Psi^*(x),\theta'=\nabla\Psi^*(y)},
		\end{aligned}
	\end{equation*}
	set $x=y$ and denote $A(x)=\left(a_{ij}(x)\right)_{i=1,j=1}^{d,d}$ with $a_{ij}(x):=\sum_{k=1}^d\partial_k\partial_i\Psi^*(x)\partial_k\partial_{j}\Psi^*(x)$. Since $\Psi^*$ is convex, then $\nabla^2\Psi^*$ is non-negative definite and so is $A(x)=\nabla^2\Psi^*(x)\nabla^2\Psi^*(x)$, and then $|a_{ij}(x)|\leq \frac{a_{ii}+a_{jj}}{2}$~(since $a_{ii}a_{jj}-a_{ij}^2\geq 0$), so we have 
	\begin{equation}
		\begin{aligned}
			\norm{\nabla k_{\Psi}(x,\cdot)}^2_{\cH}&=\sum_{i=1}^d\langle \partial_{x_{i}}k(\nabla\Psi^*(x),\cdot),\partial_{x_{i}}k(\nabla\Psi^*(x),\cdot)\rangle_{\cH_0}\\
			&\leq \frac{1}{2}B^2\sum_{j,j'}\left(a_{jj}(x)+a_{j'j'}(x)\right)\\
			&=B^2d\operatorname{Tr}(A(x))\\
			&=B^2d\operatorname{Tr}(\nabla^2\Psi^*(x)\nabla^2\Psi^*(x))\\
			&\leq \frac{B^2d^2}{{\K}^2}.
		\end{aligned}
	\end{equation}
	\begin{proof}[\textbf{Proof of \Cref{lem:A3impliesA3}}]
		Denote $\Phi(x):=k_{\Psi}(x,\cdot)\in \mathcal{H}$. Then by definition of the Mirrored Stein Fisher information and integration bay parts, we have the following: 
		\begin{equation*}
			\begin{aligned}
				I_{\Stein}({\mu}_n\mid {\pi})^{\frac{1}{2}} 
				&=\left\|\mathbb{E}_{\bX \sim \bar{\mu}_n} \big[ (\nabla V(\bX) \Phi(\bX)-\nabla \Phi(\bX)) \big]\right\|_{\mathcal{H}}\\ 
				& \leq \mathbb{E}_{\bX \sim \bar{\mu}_n}\big[\|\nabla V(\bX) \Phi(\bX)-\nabla \Phi(\bX)\big]\|_{\mathcal{H}} \\ 
				& \leq \mathbb{E}_{\bX \sim \bar{\mu}_n}\big[\|\nabla V(\bX) \Phi(\bX)\|_{\mathcal{H}}\big]
				+ \mathbb{E}_{\bX \sim \bar{\mu}_n}\big[\|\nabla \Phi(\bX)\|_{\mathcal{H}}\big] \\ 
				&=\mathbb{E}_{\bX \sim \bar{\mu}_n}\big[\|\nabla V(\bX)\|\|\Phi(\bX)\|_{\mathcal{H}}\big] + \mathbb{E}_{\bX \sim \bar{\mu}_n}
				\big[\|\nabla \Phi(\bX)\|_{\mathcal{H}} \big] \\
				& \leq B_1\mathbb{E}_{\bX \sim \bar{\mu}_n}\big[\|\nabla V(\bX)\|\big]+\frac{B_2d}{K}.
			\end{aligned}
		\end{equation*}
	\end{proof}

	\begin{proof}[\textbf{Proof of \Cref{lem:exp-grad}}]
		\begin{equation}
			\begin{aligned}
				\EE_{\bX\sim\bar{\mu}_n}[\norm{\nabla V(\bX)}]&\leq C_{p}\EE_{\bX\sim\bar{\mu}_n}[\norm{\bX}^p]+C_p\\
				&=C_pW_p^p(\bar{\mu}_n,\delta_0)+C_p\\
				&\leq C_p\left(W_p(\bar{\mu}_n,\bar{\pi})+W_p(\bar{\pi},\bar{\mu}_0)+W_p(\bar{\mu}_0,\delta_0)\right)^p+C_p\\
				&\leq C_p\left(C_{\bar{\pi},p}\left(\KL(\bar{\mu}_n\mid\bar{\pi})^{\frac{1}{p}}+\left(\frac{\KL(\bar{\mu}_n\mid\bar{\pi})}{2}\right)^{\frac{1}{2p}}\right)\right.\\
				&\left.\quad+C_{\bar{\pi},p}\left(\KL(\bar{\mu}_0\mid\bar{\pi})^{\frac{1}{p}}+\left(\frac{\KL(\bar{\mu}_0\mid\bar{\pi})}{2}\right)^{\frac{1}{2p}}\right)+W_p(\bar{\mu}_0,\delta_0)\right)^p+C_p\\
				&=C_p\left(C_{\bar{\pi},p}\left(G_p(\KL(\bar{\mu}_n\mid\bar{\pi}))+G_p(\KL(\bar{\mu}_0\mid\bar{\pi}))\right)+W_p(\bar{\mu}_0,\delta_0)\right)^p+C_p.
			\end{aligned}
		\end{equation}
	\end{proof}

	\begin{proof}[\textbf{Proof of \Cref{lem:diff}}]
		\label{proof:lem:diff}
		%\Avo{check the notation. In particular why there is g'. In \cite{salim2021complexity} it is written with g. The problem might be the inclusion map. $S_{\mu}$} 
		The proof is based on the reproducing property and Cauchy-Schwarz inequality in the RKHS space. 
		Indeed,
		% \begin{equation}
			% 	\|g(x)\|^{2}=\sum_{i=1}^{a}\left\langle k(x, .), g_{i}^{\prime}\right\rangle_{\mathcal{H}_{0}}^{2} \leq\|k(x, .)\|_{\mathcal{H}_{0}}^{2}\left\|g\right\|_{\mathcal{H}}^{2} \leq B^{2} \left\|g\right\|_{\mathcal{H}}^{2} 
			% \end{equation}
		% Similarly,
		\begin{equation*}
			\begin{aligned}
				\|J h(x)\|_{\rm H S}^{2} &= \sum_{i, j=1}^{d}\left|\frac{\partial h_{i}(x)}{\partial x_{j}}\right|^{2}\\
				&=\sum_{i, j=1}^{d}\left\langle\partial_{x_{j}} k_{\Psi}(x, .), h_{i}\right\rangle_{\mathcal{H}_{0}} \\
				&\leq \sum_{i, j=1}^{d}\left\|\partial_{x_{j}} k_{\Psi}(x, .)\right\|_{\mathcal{H}_{0}}^{2}\left\|h_{i}\right\|_{\mathcal{H}_{0}}^{2} \\
				&=\|\nabla k_{\Psi}(x, .)\|_{\mathcal{H}}^{2}\left\|h\right\|_{\mathcal{H}}^{2} \\
				&\leq \frac{B^2d^2}{{\K}^2}\left\|h\right\|_{\mathcal{H}}^{2}   .
			\end{aligned}
		\end{equation*}
		This concludes the proof.
	\end{proof}

	\section{Main Theory}

	\begin{proof}[\textbf{Proof of \Cref{thm:mainl}}]
		%	Let us first prove that the sequence $\KL (\bar{\mu}_n \mid \bar{\pi})$ is monotonically decreasing. We will use the method of 
		%	mathematical induction. 
		From \Cref{eq:gamma-cond} and \Cref{lem:A3impliesA3}, we have 
		\begin{equation}\hspace{-13mm}
			\begin{aligned}
				&\min\left\{\frac{(\alpha-1){\K}}{\alpha B_2d\|g_{\bar{\mu}_n}\|_{\cH}},\frac{1}{B_1\|g_{\bar{\mu}_n}\|_{\cH}L_1},\frac{{\K}^2}{\left(\alpha^2B_2^2d^2+{\K}^2(e-1)B_1^2A_n\right)}\right\}\\
				&\geq \min\left\{\min\left\{\frac{1}{B_1L_1},\frac{(\alpha-1){\K}}{\alpha B_2d}\right\}\frac{K}{KB_1\mathbb{E}_{\bX \sim \bar{\mu}_n}\big[\|\nabla V(\bX)\|\big]+B_2d},\frac{{\K}^2}{\alpha^2B_2^2d^2+{\K}^2B_1^2(e-1)\left(L_1\mathbb{E}_{\bX \sim \bar{\mu}_n}\big[\|\nabla V(\bX)\|\big]+L_0\right)}\right\}\\
				&=:\MM\left(\mathbb{E}_{\bX \sim \bar{\mu}_n}\big[\|\nabla V(\bX)\|\big]\right),
			\end{aligned}
		\end{equation}
		so $\MM\left(\cdot\right)$ is a non-increasing function. By \Cref{lem:exp-grad}, we further have 
		\begin{equation}
			\MM\left(\mathbb{E}_{\bX \sim \bar{\mu}_n}\big[\|\nabla V(\bX)\|\big]\right)\geq \MM\left(C_p\left(C_{\bar{\pi},p}\left(G_p(\KL(\bar{\mu}_n\mid\bar{\pi}))+G_p(\KL(\bar{\mu}_0\mid\bar{\pi}))\right)+W_p(\bar{\mu}_0,\delta_0)\right)^p+C_p\right).
		\end{equation}
		From now on, we will use mathematical induction. First,
		\begin{equation}
			\begin{aligned}
				\gamma	&\leq\MM\left(C_p\left(2C_{\bar{\pi},p}G_p(\KL(\bar{\mu}_0\mid\bar{\pi}))+W_p(\bar{\mu}_0,\delta_0)\right)^p+C_p\right)\\
				&\leq \min\left\{\frac{(\alpha-1){\K}}{\alpha B_2d\|g_{\bar{\mu}_n}\|_{\cH}},\frac{1}{B_1\|g_{\bar{\mu}_n}\|_{\cH}L_1},\frac{{\K}^2}{\left(\alpha^2B_2^2d^2+{\K}^2(e-1)B_1^2A_n\right)}\right\},
			\end{aligned}
		\end{equation}
		so it satisfies the condition of \Cref{propl-descent-lemma} and we will have $\KL(\bar{\mu}_1\mid\bar{\pi})\leq\KL(\bar{\mu}_0\mid\bar{\pi})$. Next, suppose for integer $n\geq 1$, we have $\gamma	\leq\MM\left(C_p\left(2C_{\bar{\pi},p}G_p(\KL(\bar{\mu}_0\mid\bar{\pi}))+W_p(\bar{\mu}_0,\delta_0)\right)^p+C_p\right)$ satisfying the condition of \Cref{propl-descent-lemma} for $(\bar{\mu}_k)_{k=0}^n$~(and so the sequence $\left(\KL(\bar{\mu}_k\mid\bar{\pi})\right)_{k=0}^n$ non-increasing), we need to prove $\gamma	\leq\MM\left(C_p\left(2C_{\bar{\pi},p}G_p(\KL(\bar{\mu}_0\mid\bar{\pi}))+W_p(\bar{\mu}_0,\delta_0)\right)^p+C_p\right)$ also satisfies the condition of \Cref{propl-descent-lemma}, that is 
		\begin{equation}\label{eq:increase}
			\begin{aligned}
				\gamma	&\leq\MM\left(C_p\left(C_{\bar{\pi},p}\left(G_p(\KL(\bar{\mu}_n\mid\bar{\pi}))+G_p(\KL(\bar{\mu}_0\mid\bar{\pi}))\right)+W_p(\bar{\mu}_0,\delta_0)\right)^p+C_p\right)\\
				&\leq \min\left\{\frac{(\alpha-1){\K}}{\alpha B_2d\|g_{\bar{\mu}_n}\|_{\cH}},\frac{1}{B_1\|g_{\bar{\mu}_n}\|_{\cH}L_1},\frac{{\K}^2}{\left(\alpha^2B_2^2d^2+{\K}^2(e-1)B_1^2A_n\right)}\right\}
			\end{aligned}
		\end{equation}
		\eqref{eq:increase} is due to the non-increasing property of sequence $\left(\KL(\bar{\mu}_k\mid\bar{\pi})\right)_{k=0}^n$,
		\begin{equation}
			\begin{aligned}
				\gamma	&\leq\MM\left(C_p\left(2C_{\bar{\pi},p}G_p(\KL(\bar{\mu}_0\mid\bar{\pi}))+W_p(\bar{\mu}_0,\delta_0)\right)^p+C_p\right)\\
				&\leq \MM\left(C_p\left(C_{\bar{\pi},p}\left(G_p(\KL(\bar{\mu}_1\mid\bar{\pi}))+G_p(\KL(\bar{\mu}_0\mid\bar{\pi}))\right)+W_p(\bar{\mu}_0,\delta_0)\right)^p+C_p\right)\\
				&\qquad\vdots\\
				&\leq \MM\left(C_p\left(C_{\bar{\pi},p}\left(G_p(\KL(\bar{\mu}_n\mid\bar{\pi}))+G_p(\KL(\bar{\mu}_0\mid\bar{\pi}))\right)+W_p(\bar{\mu}_0,\delta_0)\right)^p+C_p\right)\\
				&\leq \min\left\{\frac{(\alpha-1){\K}}{\alpha B_2d\|g_{\bar{\mu}_n}\|_{\cH}},\frac{1}{B_1\|g_{\bar{\mu}_n}\|_{\cH}L_1},\frac{{\K}^2}{\left(\alpha^2B_2^2d^2+{\K}^2(e-1)B_1^2A_n\right)}\right\}.
			\end{aligned}
		\end{equation}
		So by \Cref{propl-descent-lemma}, we have $\left(\KL(\bar{\mu}_k\mid\bar{\pi})\right)_{k=0}^{n+1}$ non-increasing. By mathematical induction, we proved the theorem.
	\end{proof}

	\subsection{Complexity Analysis}
	In this section, we denote $\Pi$ as the area of unit radius circle.
	Assume $\bar{\mu}_0=\cN(0,{\rm I}_d)$, we need first calculate the value of $\int_{x\in\R^d}\norm{x}^{p+1}\bar{\mu}_0(x)dx$.
	\begin{lemma}\label{lem:calcu}
		Assume $\bar{\mu}_0=\cN(0,{\rm I}_d)$, then for any $p>-d-1$ we have 
		\begin{equation}
			\int_{x\in\R^d}\norm{x}^{p+1}\bar{\mu}_0(x)dx=2^{(p+1)/2}\frac{\Gamma(\frac{p+d+1}{2})}{\Gamma(\frac{d}{2})}.
		\end{equation}
	\end{lemma}
	\begin{proof}We integrate this in polar coordinate $\left(r,\varphi_{1},\ldots,\varphi_{d-1}\right)$.
		\begin{equation}\hspace{-5mm}
			\begin{aligned}
				&\int_{x\in\R^d}\norm{x}^{p+1}\bar{\mu}_0(x)dx\\
				&=(2\Pi)^{-d/2}\int_{x\in\R^d}\norm{x}^{p+1}\exp(-\normsq{x}/2)dx\\
				&=(2\Pi)^{-d/2}\int_{0}^\infty\int_{0}^{\Pi}\cdots\int_{0}^{\Pi}\int_{0}^{2\Pi}r^{p+1}\exp(-r^2/2)r^{d-1} \sin ^{d-2}\left(\varphi_{1}\right) \sin ^{d-3}\left(\varphi_{2}\right) \cdots \sin \left(\varphi_{d-2}\right) d r d \varphi_{1} d \varphi_{2} \cdots d \varphi_{d-1}\\
				&=(2\Pi)^{-d/2}\frac{\sqrt{2\Pi}}{2}\frac{1}{\sqrt{2\Pi}}\int_{-\infty}^{\infty}r^{p+d}\exp(-r^2/2)dr\int_{0}^{\Pi}\cdots\int_{0}^{\Pi}\int_{0}^{2\Pi} \sin ^{d-2}\left(\varphi_{1}\right) \sin ^{d-3}\left(\varphi_{2}\right) \cdots \sin \left(\varphi_{d-2}\right) d \varphi_{1} d \varphi_{2} \cdots d \varphi_{d-1}\\
				&=\frac{(2\Pi)^{-(d-1)/2}}{2}M_{p+d}|S_{d-1}|\\
				&=\frac{(2\Pi)^{-(d-1)/2}}{2}M_{p+d}\frac{2 \Pi^{\frac{d}{2}}}{\Gamma\left(\frac{d}{2}\right)}\\
				&= \frac{(2\Pi)^{-(d-1)/2}}{2}\frac{2^{(p+d)/2}\Gamma\left(\frac{p+d+1}{2}\right)}{\sqrt{\Pi}}\frac{2 \Pi^{\frac{d}{2}}}{\Gamma\left(\frac{d}{2}\right)}\\
				&=2^{(p+1)/2}\frac{\Gamma(\frac{p+d+1}{2})}{\Gamma(\frac{d}{2})},
			\end{aligned}
		\end{equation}
		where $|S_{d-1}|$ is the $d-1$-dimensional volume of the sphere $S_{d-1}$, $M_{p+1}:=\int_{x\in\R} |x|^{p+1} \frac{1}{\sqrt{2\Pi}}e^{-\frac{x^2}{2}}dx$ and from \cite{winkelbauer2012moments}, we know the $p+1$-th central absolute moment is $M_{p+1}=\frac{2^{(p+1)/2}\Gamma(\frac{p+2}{2})}{\sqrt{\Pi}}=\cO(\sqrt{2p}\left(\frac{p}{e}\right)^{\frac{p}{2}})$. 
	\end{proof}
	Assume $\bar{\pi}(x)=e^{-V(x)}$, then we have the following lemma.
	\begin{lemma}\label{lem:upperboundmu0}
		Let \Cref{poly-assumption} hold and $\bar{\mu}_0=\cN(0,{\rm I}_d)$, we then have
		\begin{equation}
			\KL(\bar{\mu}_0\mid\bar{\pi})\leq \frac{d}{2}\log(\frac{1}{2\Pi e})+V(0)+\frac{C_p}{p+1}\left(2^{(p+1)/2}\frac{\Gamma(\frac{p+d+1}{2})}{\Gamma(\frac{d}{2})}+(p+1)\sqrt{2}\frac{\Gamma(\frac{d+1}{2})}{\Gamma(\frac{d}{2})}\right).
		\end{equation}
	\end{lemma}
	\begin{proof}
		By\Cref{poly-assumption}, we know 
		\begin{equation}
			\big\|\nabla V(x)\big\| \leq  C_p\left(\norm{x}^p+1\right),
		\end{equation}
		so 
		\begin{equation}\label{eq:aaa}
			\begin{aligned}
				V(x)&=\int_{0}^1\inner{\nabla V(tx)}{x}dt+V(0)\\
				&\leq \int_{0}^1\norm{\nabla V(tx)}\norm{x}dt+V(0)\\
				&\leq \int_{0}^1C_p\left(\norm{tx}^p+1\right)\norm{x}dt+V(0)\\
				&\leq C_p\left(\frac{\norm{x}^{p+1}}{p+1}+\norm{x}\right)+V(0).
			\end{aligned}
		\end{equation}
		We procced to calculate $\KL(\bar{\mu}_0\mid\bar{\pi})$,
		\begin{equation}
			\begin{aligned}
				\KL(\bar{\mu}_0\mid\bar{\pi})&=\int_{x\in\R^d}\log(\frac{\bar{\mu}_0}{\bar{\pi}})(x)\bar{\mu}_0(x)dx\\
				&=\int_{x\in\R^d} \log(\bar{\mu}_0)(x)\bar{\mu}_0(x)dx+\int_{x\in\R^d} V(x)\bar{\mu}_0(x)dx\\
				&\leq\frac{d}{2}\log(\frac{1}{2\Pi e})+\int_{x\in\R^d} \left(C_p\left(\frac{\norm{x}^{p+1}}{p+1}+\norm{x}\right)+V(0)\right)\bar{\mu}_0(x)dx\\
				&=  \frac{d}{2}\log(\frac{1}{2\Pi e})+V(0)+\frac{C_p}{p+1}\int_{x\in\R^d}{\norm{x}^{p+1}}+(p+1)\norm{x}\bar{\mu}_0(x)dx\\
				&= \frac{d}{2}\log(\frac{1}{2\Pi e})+V(0)+\frac{C_p}{p+1}\left(2^{(p+1)/2}\frac{\Gamma(\frac{p+d+1}{2})}{\Gamma(\frac{d}{2})}+(p+1)\sqrt{2}\frac{\Gamma(\frac{d+1}{2})}{\Gamma(\frac{d}{2})}\right)
			\end{aligned}
		\end{equation}
		where we use \Cref{lem:calcu}.
	\end{proof}
	According to \Cref{lem:upperboundmu0}, the order of  
	$\KL(\bar{\mu}_0\mid \bar{\pi})$ is
	\begin{equation}\label{eq:klklbound}
		\KL(\bar{\mu}_{0}\mid\bar{\pi})=\tilde{\mathcal{O}}\left(\frac{2^{(p+1)/2}\Gamma(\frac{p+d+1}{2})}{(p+1)\Gamma(\frac{d}{2})}\right).
	\end{equation}
	
	%\begin{theorem}
	%	Let \Cref{Al-ker-bound}, \Cref{Al-l0l1}, \Cref{Tpl-assumption} and \Cref{poly-assumption} hold and let $\bar{\mu}_0=\cN(0,{\rm I}_d)$,  then
	%	\begin{equation} \label{eq:interatincomplexity}
		%	n=\Omega\left(\frac{2^{(p+1)/2}L_0C_p\Gamma(\frac{p+d+1}{2})}{(p+1)\Gamma(\frac{d}{2})\varepsilon}+
		%	\frac{8^{p}(L_1+1)C_{\bar{\pi},p}^pC_p^3\Gamma(\frac{p+d+1}{2})^2}{(p+1)^2\Gamma(\frac{d}{2})^2\varepsilon}\right)
		%	\end{equation}
	%	iterations of SVGD are sufficient to output $\bar{\mu}$ such that $I_{{Stein }}(\bar{\mu} \mid \bar{\pi}) \leq \varepsilon$.
	%\end{theorem}
	
	\begin{proof}[\textbf{Proof of \Cref{cor:convergence}}]
		By \Cref{thm:mainl}, we know that
		\begin{equation*}
			\frac{1}{n} \sum_{k=0}^{n-1} I_{\Stein}\left(\bar{\mu}_{k} \mid \bar{\pi}\right) \leq\frac{2\KL(\bar{\mu}_0\mid\bar{\pi})}{n\gamma},
		\end{equation*}
		we need to estimate the order of $\frac{\KL(\bar{\mu}_0\mid\bar{\pi})}{\gamma}$.  Firstly, from \Cref{lem:upperboundmu0}  we know $\KL(\bar{\mu}_0\mid\bar{\pi})={\cO}((C_p+1)dM_p)$, so
		\begin{equation}
			\begin{aligned}
				G_p(\KL(\bar{\mu}_0\mid\bar{\pi}))&=\left(\KL(\bar{\mu}_0\mid\bar{\pi})^{\frac{1}{p}}+\left(\frac{\KL(\bar{\mu}_0\mid\bar{\pi})}{2}\right)^{\frac{1}{2p}}\right)\\
				&= {\cO}(\KL(\bar{\mu}_0\mid\bar{\pi})^{\frac{1}{p}}).
			\end{aligned}	
		\end{equation}
		
		Next, we estimate the order of $\frac{1}{\gamma}$, we have 
		\begin{equation}\label{eq:gang} 
			\begin{aligned}
				\frac{1}{\gamma}&= \frac{1}{\MM\left(C_p\left(2C_{\bar{\pi},p}G_p(\KL(\bar{\mu}_0\mid\bar{\pi}))+W_p(\bar{\mu}_0,\delta_0)\right)^p+C_p\right)}\\
				&=\tilde{\mathcal{O}}(\left(2C_{\bar{\pi},p}G_p(\KL(\bar{\mu}_0\mid\bar{\pi}))+W_p(\bar{\mu}_0,\delta_0)\right)^p)\\
				&=\tilde{\mathcal{O}}\left(2C_{\bar{\pi},p}\left(\frac{2^{(p+1)/2}\Gamma(\frac{p+d+1}{2})}{(p+1)\Gamma(\frac{d}{2})}\right)^{\frac{1}{p}}+\left(2^{\frac{p}{2}}\frac{\Gamma(\frac{p+d}{2})}{\Gamma(\frac{d}{2})}\right)^{\frac{1}{p}}\right)^p\\
				&=\tilde{\mathcal{O}}\left(\frac{32^{\frac{p}{2}}C_{\bar{\pi},p}^p\Gamma(\frac{p+d+1}{2})}{(p+1)\Gamma(\frac{d}{2})}\right).
			\end{aligned}
		\end{equation}
		
		So we finally have 
		\begin{equation}
			\frac{\KL(\bar{\mu}_0\mid\bar{\pi})}{\gamma}=\tilde{\mathcal{O}}\left(
			\frac{8^{p}C_{\bar{\pi},p}^p\Gamma(\frac{p+d+1}{2})^2}{(p+1)^2\Gamma(\frac{d}{2})^2}\right),
		\end{equation}
		to get $\frac{2\KL(\bar{\mu}_0\mid\bar{\pi})}{n\gamma}\leq\varepsilon$, we need $n\geq\frac{2\KL(\bar{\mu}_0\mid\bar{\pi})}{\gamma\varepsilon}$, that is 
		\begin{equation}
			n=\tilde{\Omega}\left(
			\frac{8^{p}C_{\bar{\pi},p}^p\Gamma(\frac{p+d+1}{2})^2}{(p+1)^2\Gamma(\frac{d}{2})^2\varepsilon}\right).
		\end{equation}
	\end{proof}
	
	By \text { the Stirling formula } $\Gamma(z+1) \sim \sqrt{2 \Pi z}\left(\frac{z}{e}\right)^{z}$, we know 
	\begin{equation}\label{eq:firstdimbound}
		\begin{aligned}
			\frac{\Gamma(\frac{p+d+1}{2})^2}{\Gamma(\frac{d}{2})^2}
			&=\Omega\left(\frac{{\Pi(d+p-1)}\left(\frac{d+p-1}{2e}\right)^{{d+p-1}}}{{\Pi(d-2)}\left(\frac{d-2}{2e}\right)^{{d-2}}}\right)\\
			&=\Omega\left(\left(\frac{d+p-1}{d-2}\right)^{d-2}\left(\frac{d+p-1}{2e}\right)^{p+1}\right)\\
			&=\Omega\left(d^{p+1}\right), \quad \text{where we only consider the order of $d$},
		\end{aligned}
	\end{equation}
	so the iteration complexity has dimension dependency of order $p+1$.

	\section{The Cases under $\text{T}_p$ inequality with $1\leq p\leq 2$}\label{sec:Tpcase}
%	In this section, we assume $\bar{\pi}$ satisfies Talagrand's $\text{T}_p$ inequality \eqref{eq:Tp} instead of \Cref{Tpl-assumption} and keep the other assumptions the same. When $p\in [1,2]$, \Cref{eq:Tp} is more preferable in our analysis since $W_p(\bar{\mu},\bar{\pi})=\cO(\KL(\bar{\mu}\mid\bar{\pi})^{\frac{1}{2}})$ while \Cref{eq:generalTp} gives the bound $W_p(\bar{\mu},\bar{\pi})=\cO(\KL(\bar{\mu}\mid\bar{\pi})^{\frac{1}{p}})$ which is bigger.We focus on the cases when $p\in [1,2]$ and we give a modified version of \Cref{lem:exp-grad}, \Cref{thm:mainl} and \Cref{cor:convergence}.
%	
%	\begin{lemma}[\textbf{modified \Cref{lem:exp-grad} under $\text{T}_p$ inequality with $1\leq p\leq 2$}]\label{mlem:exp-grad}
%		Suppose that the potential function $V$ satisfies the Assumptions \ref{Al-l0l1}, \ref{poly-assumption} and \Cref{eq:Tp} with constants $L_0$, $L_1$, $p$. Then,
%		\begin{equation}\label{meq:muupper}
%			%		\EE_{\bX\sim\bar{\mu}_n}[\|\nabla V(\bX)\|] \leq 2^{p-1}C_p W_p^p(\bar{\pi},\delta_0)+C_p+2^{p-1}C_p C_{\bar{\pi},p}^pG_p^p(\KL(\bar{\mu}_n\mid\bar{\pi}))\\
%			\EE_{\bX\sim\bar{\mu}_n}[\|\nabla V(\bX)\|] \leq C_p\left(\sqrt{\frac{2\KL(\bar{\mu}_n\mid\bar{\pi})}{\lambda}}+\sqrt{\frac{2\KL(\bar{\mu}_0\mid\bar{\pi})}{\lambda}}+W_p(\bar{\mu}_0,\delta_0)\right)^p+C_p,
%		\end{equation}
%		where $C_p$ from \Cref{poly-assumption}.
%	\end{lemma}
	\begin{proof}[\textbf{Proof of \Cref{mlem:exp-grad}}]
		\begin{equation}
			\begin{aligned}
				\EE_{\bX\sim\bar{\mu}_n}[\norm{\nabla V(\bX)}]&\leq C_{p}\EE_{\bX\sim\bar{\mu}_n}[\norm{\bX}^p]+C_p\\
				&=C_pW_p^p(\bar{\mu}_n,\delta_0)+C_p\\
				&\leq C_p\left(W_p(\bar{\mu}_n,\bar{\pi})+W_p(\bar{\pi},\bar{\mu}_0)+W_p(\bar{\mu}_0,\delta_0)\right)^p+C_p\\
				&\leq C_p\left(\sqrt{\frac{2\KL(\bar{\mu}_n\mid\bar{\pi})}{\lambda}}+\sqrt{\frac{2\KL(\bar{\mu}_0\mid\bar{\pi})}{\lambda}}+W_p(\bar{\mu}_0,\delta_0)\right)^p+C_p
			\end{aligned}
		\end{equation}
	\end{proof}
	
%	\begin{theorem}[\textbf{modified \Cref{thm:mainl} under $\text{T}_p$ inequality with $1\leq p\leq 2$}]
%		\label{mthm:mainl}
%		Let the target distribution $\bar{\pi}$ satisfies \Cref{eq:Tp} and its potential function $V$ satisfy the Assumptions \ref{Al-ker-bound}, \ref{Al-l0l1},
%		and \ref{poly-assumption}.%\aknote{Perhaps adding labels to the assumptions instead numbers, would ease the reading.}
%		~	Also, suppose that the step-size $\gamma$ satisfies 
%		\begin{equation} \label{meq:gamma-final-condl}\hspace{-10mm}
%			\gamma\leq\MM\left(C_p\left(\sqrt{\frac{2\KL(\bar{\mu}_n\mid\bar{\pi})}{\lambda}}+\sqrt{\frac{2\KL(\bar{\mu}_0\mid\bar{\pi})}{\lambda}}+W_p(\bar{\mu}_0,\delta_0)\right)^p+C_p\right),
%		\end{equation}
%		where  $\MM(x):=\min\left\{\min\left\{\frac{1}{B_1L_1},\frac{(\alpha-1){\K}}{\alpha B_2d}\right\}\frac{1}{B_1x+B_2d},\frac{{\K}^2}{\alpha^2B_2^2d^2+{\K}^2B_1^2(e-1)\left(L_1x+L_0\right)}\right\}$,
%		then the following bound is true:
%		\begin{equation}\label{eq:thm-main}
%			\mathrm{KL}\left(\bar{\mu}_{n+1} \mid \bar{\pi}\right)-\mathrm{KL}\left(\bar{\mu}_{n} \mid \bar{\pi}\right)
%			\leq - \frac{\gamma}{2} I_{Stein}(\bar{\mu}_n \mid \bar{\pi}).
%		\end{equation}
%	\end{theorem}
	\begin{proof}[\textbf{Proof of \Cref{mthm:mainl}}]
		The proof is the same as the one for \Cref{thm:mainl} but with this new bound \eqref{meq:muupper}.
	\end{proof}

%	\begin{corollary}[\textbf{modified \Cref{cor:convergence} under $\text{T}_p$ inequality with $1\leq p\leq 2$}]\label{mcor:convergence}
%		Let the target distribution $\bar{\pi}$ satisfies \Cref{eq:Tp} and its potential function $V$ satisfy the Assumptions \ref{Al-ker-bound}, \ref{Al-l0l1},
%		and \ref{poly-assumption}. If $\gamma$ satisfies condition \eqref{meq:gamma-final-condl}, then
%		we have
%		\begin{equation}
%			\frac{1}{n} \sum_{k=0}^{n-1} I_{\Stein}\left(\bar{\mu}_{k} \mid \bar{\pi}\right) 
%			\leq \frac{2 \KL (\bar{\mu}_0\mid \bar{\pi}) }{n \gamma}.
%		\end{equation}
%		If~ $\bar{\pi}(x)=\exp(-V(x))$ and let $\bar{\mu}_0=\cN(0,{\rm I}_d)$,  then to make $\frac{1}{n} \sum_{k=0}^{n-1} I_{\Stein}\left(\bar{\mu}_{k} \mid \bar{\pi}\right) \leq\varepsilon$, we need 
%		\begin{equation}\label{eq:iteration bound}
%			n=\tilde{\Omega}\left(\frac{d^{\frac{(p+2)(p+1)}{4}}}{\lambda^{\frac{p}{2}}\varepsilon}\right).
%		\end{equation}
%	\end{corollary}
	
	\begin{proof}[\textbf{Proof of \Cref{mcor:convergence}}]
		Similar to the proof of \Cref{cor:convergence}, we need to estimate the order of $\frac{\KL(\bar{\mu}_0\mid\bar{\pi})}{\gamma}$.
		By \Cref{eq:klklbound} and \text { the Stirling formula } $\Gamma(z+1) \sim \sqrt{2 \Pi z}\left(\frac{z}{e}\right)^{z}$, we know 
		\begin{equation}\label{eq:kkllbound}
			\begin{aligned}
				\KL(\bar{\mu}_{0}\mid\bar{\pi})&=\tilde{\mathcal{O}}\left(\frac{\Gamma(\frac{p+d+1}{2})}{\Gamma(\frac{d}{2})}\right)\\
				&=\tilde{\cO}\left(\frac{\sqrt{\Pi(d+p-1)}\left(\frac{d+p-1}{2e}\right)^{\frac{d+p-1}{2}}}{\sqrt{\Pi(d-2)}\left(\frac{d-2}{2e}\right)^{\frac{d-2}{2}}}\right)\\
				&=\tilde{\cO}\left(\left(\frac{d+p-1}{d-2}\right)^{\frac{d-2}{2}}\left(\frac{d+p-1}{2e}\right)^{\frac{p+1}{2}}\right)\\
				&=\tilde{\cO}\left(d^\frac{p+1}{2}\right), 
			\end{aligned}
		\end{equation}
		we also know 
		\begin{equation}
			\begin{aligned}
				W_p^p(\bar{\mu}_0,\delta_0)=\int_{x\in\R^d}\norm{x}^pd\bar{\mu}_0(x)=2^{\frac{p}{2}}\frac{\Gamma(\frac{p+d}{2})}{\Gamma{\frac{d}{2}}}=\cO\left(d^{\frac{p}{2}}\right).
			\end{aligned}
		\end{equation}
		So we have 
		\begin{equation}
			\begin{aligned}
				&\left(2\sqrt{\frac{2\KL(\bar{\mu}_0\mid\bar{\pi})}{\lambda}}+W_p(\bar{\mu}_0,\delta_0)\right)^p\\
				&=\tilde{\cO}\left(\left(\sqrt{\frac{1}{\lambda}}d^{\frac{p+1}{4}}+d^{\frac{1}{2}}\right)^p\right)\\
				&=\tilde{\cO}\left(\max\left\{\frac{1}{\lambda^{\frac{p}{2}}}d^{\frac{p(p+1)}{4}},d^{\frac{p}{2}}\right\}\right)\\
				&=\tilde{\cO}\left(\frac{d^{\frac{p(p+1)}{4}}}{\lambda^{\frac{p}{2}}}\right),
			\end{aligned}
		\end{equation}
		and then
		\begin{equation}
			\begin{aligned}
				\frac{\KL(\bar{\mu}_0\mid\bar{\pi})}{\gamma}&=\tilde{\cO}\left(\frac{d^{\frac{p(p+1)}{4}}}{\lambda^{\frac{p}{2}}}d^{\frac{p+1}{2}}\right)\\
				&=\tilde{\cO}\left(\frac{d^{\frac{(p+2)(p+1)}{4}}}{\lambda^{\frac{p}{2}}}\right).
			\end{aligned}
		\end{equation}
		To get $\frac{2\KL(\bar{\mu}_0\mid\bar{\pi})}{n\gamma}\leq\varepsilon$, we need $n\geq\frac{2\KL(\bar{\mu}_0\mid\bar{\pi})}{\gamma\varepsilon}$, that is 
		\begin{equation}
			n=\tilde{\Omega}\left(\frac{d^{\frac{(p+2)(p+1)}{4}}}{\lambda^{\frac{p}{2}}\varepsilon}\right).
		\end{equation}
	\end{proof}
	%Compare with \Cref{eq:aaasss} and \Cref{eq:iteration bound}, we know under $\text{T}_p$ inequality \eqref{eq:Tp} with $1\leq p\leq2$, iteration complexity $n$ will have slightly better dimension dependent.

\end{document}